\documentclass[11pt,final]{article}
\usepackage{a4}
\usepackage{amsmath}%
\usepackage{amstext}%
\usepackage{amssymb}%
\usepackage{epsfig}%
\usepackage{cite}
\usepackage{tikz}
\usepackage{subfig}
\usepackage{caption}
\usepackage{multirow}
\usepackage{booktabs}
\usepackage{floatrow}
\usepackage{comment}
\usepackage{listings}
\usepackage{showkeys}

\newtheorem{theorem}{Theorem}[section]

\newtheorem{lemma}[theorem]{Lemma}

\newtheorem{rem}[theorem]{Remark}
\newenvironment{remark}{\rem\rm}{\endrem}

\newcounter{unnumber}

\newtheorem{prf}[unnumber]{Proof.}
\newenvironment{proof}{\prf\rm}{\hfill{$\blacksquare$}\endprf}
\newcommand{\R}{\mathbb{R}}%
\newcommand{\N}{\mathbb{N}}%
\newcommand{\e}{\varepsilon}%
\newcommand{\ol}{\overline}%

\newcommand{\n}{{\nabla}}

\newcommand{\To}{\longrightarrow}
\def\a{\alpha}

\def\b{\beta}
\def\e{\epsilon}
\def\d{\delta}

\def\<{\langle}
\def\>{\rangle}
\DeclareMathOperator*\argmin{argmin}

\title{Tikhonov regularization of a second order dynamical system with Hessian driven damping}

\author{Radu Ioan Bo\c{t} \thanks{University of Vienna, Faculty of Mathematics, Oskar-Morgenstern-Platz 1, A-1090 Vienna, Austria,
email: radu.bot@univie.ac.at. Research partially supported by FWF (Austrian Science Fund), project I 2419-N32.} \and
Ern\"{o} Robert Csetnek \thanks {University of Vienna, Faculty of Mathematics, Oskar-Morgenstern-Platz 1, A-1090 Vienna, Austria,
email: ernoe.robert.csetnek@univie.ac.at. Research supported by FWF (Austrian Science Fund), project P 29809-N32.}
 \and Szil\'{a}rd Csaba L\'{a}szl\'{o} \thanks{Technical University of Cluj-Napoca, Department of Mathematics, Memorandumului 28, Cluj-Napoca,
 Romania, e-mail: szilard.laszlo@math.utcluj.ro. This work was supported by a grant of Ministry of Research and Innovation, CNCS - UEFISCDI, project number PN-III-P1-1.1-TE-2016-0266 and by a grant of Ministry of Research and Innovation,
 CNCS - UEFISCDI, project number PN-III-P4-ID-PCE-2016-0190, within PNCDI III.}}

\begin{document}
\maketitle

{\bf The paper is dedicated to Prof. Marco A. L\'opez on the occasion of his 70th birthday.}\vspace{1ex}

\noindent \textbf{Abstract.} We investigate the asymptotic properties of the trajectories generated by a second-order dynamical system with Hessian 
driven damping and a Tikhonov regularization term in connection with the minimization of a smooth convex function in Hilbert spaces. We obtain fast convergence 
results  for the function values along the trajectories. The Tikhonov regularization term enables the derivation of strong convergence results of the trajectory to the minimizer of the objective function of minimum norm. \vspace{1ex}

\vspace{1ex}


\section{Introduction}\label{sec-intr}

The paper of Su, Boyd and Cand\`{e}s \cite{su-boyd-candes} was the starting point of intensive research of second order dynamical systems with an asymptotically vanishing damping term of the form
\begin{equation}\label{dysy-sbc}
\ddot{x}(t)+\frac{\a}{t}\dot{x}(t)+\n g(x(t))=0, \ t \geq t_0 > 0,
\end{equation}
where $g:\mathcal{H}\To \R$ is a convex and continuously Fr\'{e}chet differentiable function defined on a real Hilbert space $\mathcal{H}$ fulfilling $\argmin g \neq \emptyset$. The aim is to approach by the trajectories generated by this system the solution set of the optimization problem 
\begin{equation}\label{opt}\min_{x\in \mathcal{H}} g(x). 
\end{equation}
The convergence rate of the objective function along the trajectory is in case $\alpha>3$ of
$$g(x(t))-\min g=o\left(\frac{1}{t^2}\right),$$
while in case $\alpha=3$ it is of
$$g(x(t))-\min g=O\left(\frac{1}{t^2}\right),$$
where $\min g \in \R$ denotes the minimal value of $g$. Also in view of this fact, system \eqref{dysy-sbc} is seen as a continuous version of the celebrated Nesterov accelerated gradient scheme (see \cite{nesterov83}). In what concerns the asymptotic properties of the generated trajectories, weak convergence to a minimizer of $g$ as 
the time goes to infinity has been proved by Attouch, Chbani, Peypouquet and Redont in \cite{att-c-p-r-math-pr2018} (see also \cite{att-ch-r-esaim2019}) for $\alpha > 3$. Without any further geometrical assumption on $g$, the convergence of the trajectories in the case $\alpha \leq 3$ is still an open problem.

Second order dynamical systems with a geometrical Hessian driven damping term have aroused the interest of the researchers, due to both their applications in optimization and mechanics and their natural relations to Newton and Levenberg-Marquardt iterative methods (see \cite{alv-att-bolte-red}). Furthermore, it has been observed for some classes of optimization problems that a geometrical damping term governed by the Hessian can induce a stabilization of the trajectories. In \cite{att-p-r-jde2016} the dynamical system with Hessian driven damping term
\begin{equation}\label{dysy-hess}
\ddot{x}(t)+\frac{\a}{t}\dot{x}(t)+\b\n^2g(x(t))\dot{x}(t)+\n g(x(t))=0, \ t \geq t_0 > 0,
\end{equation}
where $\alpha \geq 3$ and $\beta >0$, has been investigated in relation with the optimization problem \eqref{opt}. Fast convergence rates for the values and the gradient of the objective function along the trajectories are obtained and the weak convergence of the trajectories to a minimizer of $g$ is shown. We would also like to mention that iterative schemes which result via (symplectic) discretizations of dynamical systems with Hessian driven damping terms have been recently formulated and investigated from the point of view of their convergence properties in \cite{shi-du-jordan-su2018, shi-du-su-jordan2019, att-ch-fad-r-arx2019}. 

Another development having as a starting point \eqref{dysy-sbc} is the investigation of dynamical systems involving a Tikhonov regularization term. Attouch, Chbani and Riahi investigated in this context in \cite{ACR} the system
\begin{equation}\label{dysy-sbc-tikh}
\ddot{x}(t)+\frac{\a}{t}\dot{x}(t)+\n g(x(t))+\e(t)x(t)=0, \ t \geq t_0 > 0,
\end{equation}
where $\alpha \geq 3$ and $\e:[t_0,+\infty)\To [0,+\infty)$. One of the main benefits of considering such a regularized  dynamical system is that it generates trajectories which converge strongly to the minimum norm solution of \eqref{opt}. Besides that, in \cite{ACR} it was proved that the fast convergence rate of the objective function values along the trajectories remains unaltered. For more insights into the role played by the Tikhonov regularization for optimization problems and, more general, for monotone inclusion problems, we refer the reader to \cite{alv-cab, att, att-com1996, com-peyp-sor}. 

This being said, it is natural to investigate a second order dynamical system which combines a Hessian driven damping and a Tikhonov regularization term and to examine if it inherits the properties of the dynamical systems \eqref{dysy-hess} and \eqref{dysy-sbc-tikh}. This is the aim of the manuscript, namely the analysis in the framework of the general assumption stated below of the dynamical system
\begin{equation}\label{dysy}
\ddot{x}(t)+\frac{\a}{t}\dot{x}(t)+\b\n^2g(x(t))\dot{x}(t)+\n g(x(t))+\e(t)x(t)=0, \ t \geq t_0 > 0, \ x(t_0)=u_0, \ \dot{x}(t_0)=v_0,
\end{equation}
where $\a\geq 3$ and $\b \geq 0$,  and $u_0, v_0 \in {\cal H}$.\vspace{1ex}

\fbox{\parbox{0.93\textwidth}{
{\bf General assumption:} 
\begin{itemize}
\item[$\bullet$] $g:\mathcal{H}\To \R$ is a convex and twice Fr\'{e}chet differentiable function with Lipschitz continuous gradient on bounded sets and $\argmin g \neq \emptyset$;

\item[$\bullet$] $\e:[t_0,+\infty)\To[0,+\infty)$ is a nonincreasing function of class $C^1$ fulfilling $\lim_{t\To+\infty}\e(t)=0$.
\end{itemize}
}}\vspace{1ex}

The fact that the starting time $t_0$ is taken as strictly greater than zero comes from the singularity of the damping coefficient $\frac{\a}{t}$. This is not a limitation of the generality of the proposed approach, since we will focus on the asymptotic behaviour of the generated trajectories. Notice that if $\mathcal{H}$ is finite-dimensional, then the Lipschitz continuity of $\nabla g$ on bounded sets follows from the continuity of $\nabla^2 g$.

To which extent the Tikhonov regularization does influence the convergence behaviour of the trajectories generated by \eqref{dysy} can be seen even when minimizing  a one dimensional function. Consider the convex and twice continuously differentiable function 
\begin{align}\label{functiong}
g : \R \rightarrow \R, \quad g(x)=\left\{
\begin{array}{ll}
-(x+1)^3, & \mbox {if } x < -1\\
0, & \mbox{ if } -1\leq x\leq 1\\
 (x-1)^3, & \mbox{ if } x > 1.
\end{array}\right.
\end{align}

\begin{figure}[h]
	\centering
	\captionsetup[subfigure]{position=top}
	%\subfloat[...]
	{\includegraphics*[viewport=130 250 485 574,width=0.46\textwidth]{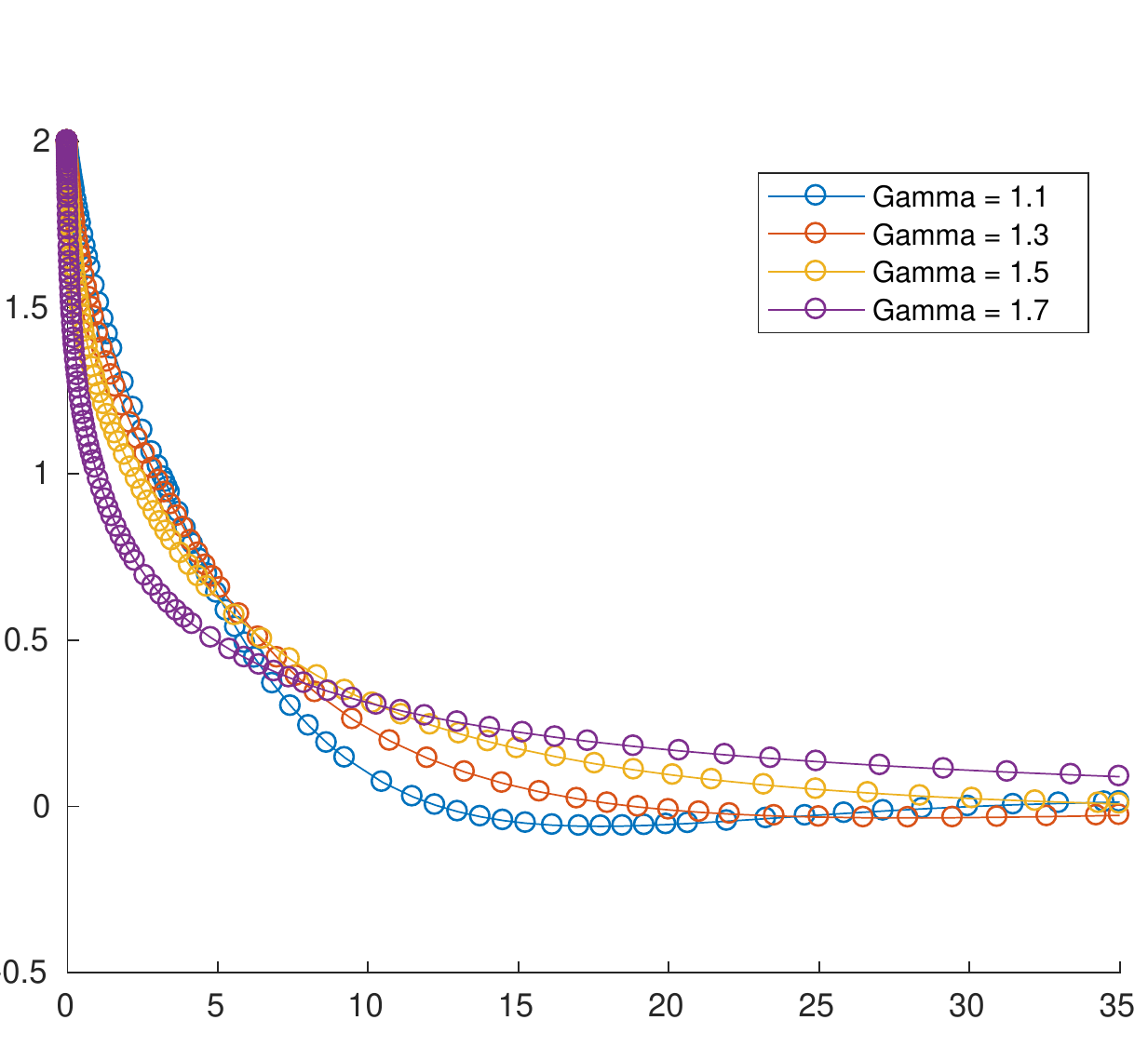}} \hspace{0.8cm}
	%\subfloat[...]
	{\includegraphics*[viewport=130 250 485 574,width=0.46\textwidth]{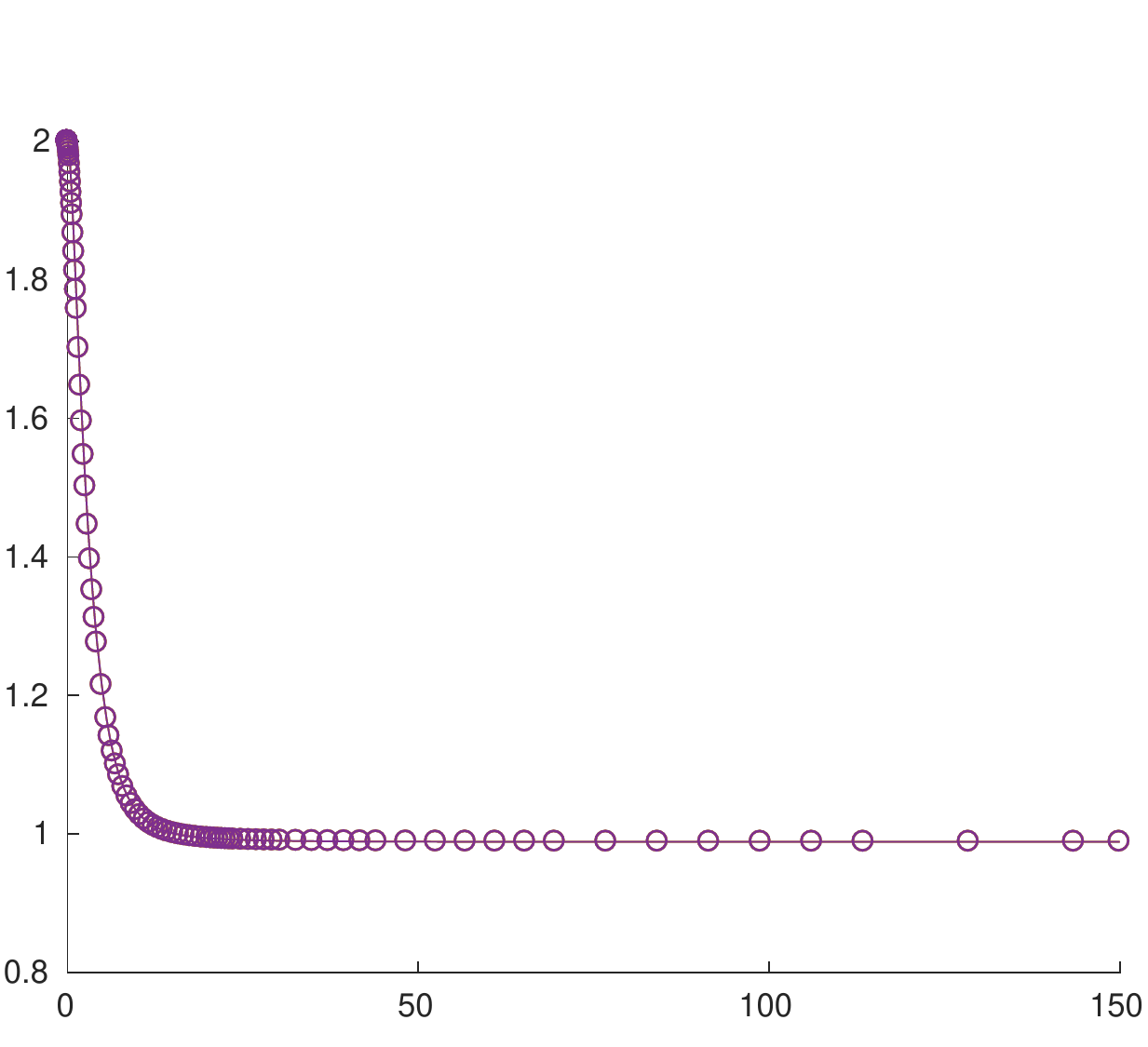}}\\
	%\subfloat[...]
	{\includegraphics*[viewport=130 250 485 574,width=0.46\textwidth]{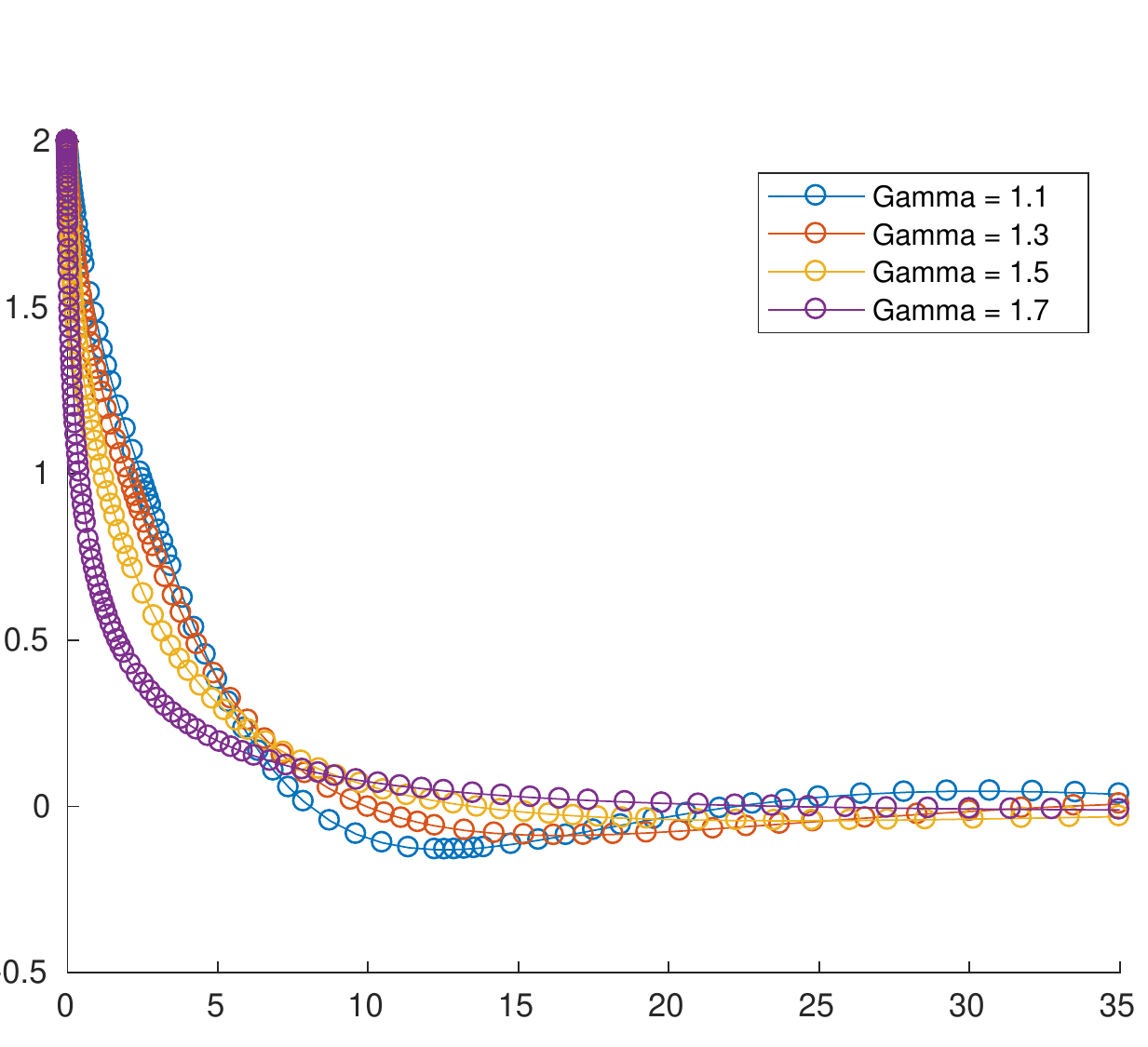}} \hspace{0.8cm}
	%\subfloat[...]
	{\includegraphics*[viewport=130 250 485 574,width=0.46\textwidth]{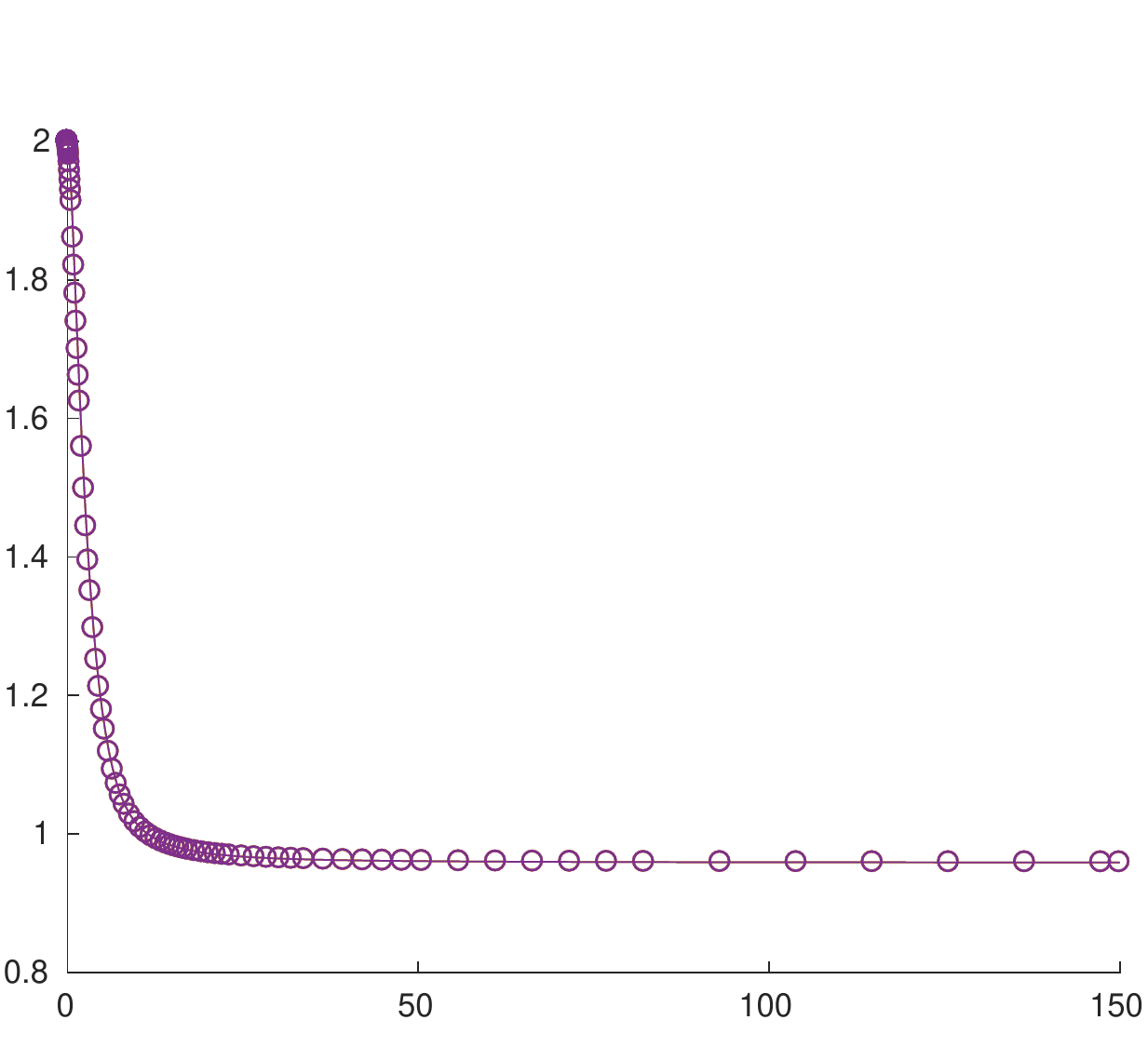}}
	%\subfloat[...]
	\caption{\small First column: the trajectories of the dynamical system with Tikhonov regularization $\e(t)=t^{-\gamma}$ are approaching the minimum norm solution $x^*=0$.  Second column: the trajectories of the dynamical system without Tikhonov regularization the trajectory are approaching the optimal solution $1$.}
	\label{fig:ex1}	
\end{figure}

It holds that $\argmin g=[-1,1]$ and $x^*=0$ is its minimum norm solution. In the second column of 
Figure \ref{fig:ex1} we can see the behaviour of the trajectories generated by the dynamical system without Tikhonov regularization (which corresponds to the case when $\e$ is identically $0$) for $\beta =1$ and $\alpha =3$ and, respectively, $\alpha =4$. In both cases the trajectories are approaching the optimal solution $1$, which is a minimizer of $g$, however, not the minimum norm solution.

In the first column of Figure \ref{fig:ex1} we can see the behaviour of the trajectories generated by the dynamical system with Tikhonov  parametrizations of the form  $t \mapsto \e(t)=t^{-\gamma}$, for different values of $\gamma \in (1,2)$, which is in accordance to the conditions in
Theorem \ref{strongconvergence}, $\beta =1$ and $\alpha =3$ and, respectively, $\alpha =4$. The trajectories are approaching the minimum norm solution $x^*=0$.

The organization of the paper is as follows. We start the analysis of the dynamical system \eqref{dysy} by proving the existence and uniqueness of a global $C^2$-solution.  In the third section we provide two different settings for the Tikhonov parametrization $t \mapsto \e(t)$ in both of which $g(x(t))$ converges to $\min g$, the minimal value of $g$, with a convergence rate of $O\left(\frac{1}{t^2}\right)$ for $\alpha =3$ and of $o\left(\frac{1}{t^2}\right)$ for $\alpha >3$. The proof relies on Lyapunov theory; the choice of the right energy functional plays a decisive role in this context. Weak convergence of the trajectory is also derived for $\alpha >3$. In the last section we focus on the proof of strong convergence to a minimum norm solution: firstly, in a general setting, for the ergodic trajectory, and, secondly, in a slightly restrictive setting, for the trajectory $x(t)$ itself.

\section{Existence and uniqueness}\label{sec2}

In this section we will prove the existence and uniqueness of a global $C^2$-solution of the dynamical system \eqref{dysy}. The proof of the existence and uniqueness theorem is based on the idea to reformulate \eqref{dysy} as a particular first order dynamical system in a suitably chosen product space (see also \cite{att-p-r-jde2016}). 

\begin{theorem}\label{rewrite} For every initial value $(u_0, v_0) \in \mathcal{H}\times \mathcal{H}$, there exists a unique global $C^2$-solution $x:[t_0, + \infty) \to \mathcal{H}$ to \eqref{dysy}.
\end{theorem}

\begin{proof}
Let $(u_0, v_0) \in \mathcal{H}\times \mathcal{H}$. First we assume that $\beta =0$, which gives the dynamical system \eqref{dysy-sbc-tikh} investigated in \cite{ACR}. The statement follows from \cite[Proposition 2.2(b)]{cabotenglergadat} (see also the discussion in \cite[Section 2]{ACR}).

Assume now that $\beta >0$.  We notice that $x:[t_0,+\infty)\To\mathcal{H}$ is a solution of the dynamical system \eqref{dysy}, that is
$$
\ddot{x}(t)+\frac{\a}{t}\dot{x}(t)+\b\n^2g(x(t))\dot{x}(t)+\n g(x(t))+\e(t)x(t)=0,\,\,x(t_0)=u_0,\,\dot{x}(t_0)=v_0,
$$
if and only if $(x,y) :[t_0,+\infty)\To\mathcal{H} \times \mathcal{H}$ is a solution of the dynamical system
$$\left\{\begin{array}{lll}
\dot{x}(t)+\b\n g(x(t))-y(t)=0\\
\dot{y}(t)+\frac{\a}{t}\dot{x}(t)+\n g(x(t))+\e(t)x(t)=0\\
 x(t_0)=u_0,\,y(t_0)=v_0+\b\n g(u_0),
\end{array}
\right.
$$
which is further equivalent to
\begin{equation}\label{dynzx}
\left\{\begin{array}{lll}
\dot{x}(t)+\b\n g(x(t))-y(t)=0\\
\dot{y}(t) + \frac{\a}{t} y(t) + \left(1- \frac{\a\b}{t}\right) \n g(x(t)) +\e(t)x(t)=0\\
 x(t_0)=u_0,\,y(t_0)=v_0+\b\n g(u_0).
\end{array}
\right.
\end{equation}
We define $F:[t_0,+\infty) \times \mathcal{H} \times \mathcal{H} \rightarrow \mathcal{H} \times \mathcal{H}$ by
$$F(t,u,v) = \left (-\beta \n g(u) + v,  - \frac{\a}{t} v - \left(1- \frac{\a\b}{t}\right) \n g(u) - \e(t)u\right),$$
and write \eqref{dynzx} as
\begin{equation}\label{dynzx2}
\left\{\begin{array}{lll}
\big(\dot x(t), \dot y(t)\big) = F(t,x(t),y(t))\\
 \big(x(t_0), y(t_0)\big)=\big(u_0, v_0+\b\n g(u_0)\big).
\end{array}
\right.
\end{equation}
Since $\n g$ is Lipschitz continuous on bounded sets and continuously differentiable, the local existence and uniqueness theorem (see \cite[Theorem 46.2 and Theorem 46.3]{sellyou}) guarantees the existence of a unique solution $(x,y)$ of \eqref{dynzx2} defined on a maximum intervall $[t_0,T_{\max})$, where $t_0 < T_{\max}\le+\infty.$ Furthermore, either $T_{\max} = +\infty$ or $\lim_{t \rightarrow T_{\max}} \|x(t)\| + \|y(t)\| = +\infty$. We will prove that $T_{\max}=+\infty$, which will imply that $x$ is the unique global $C^2$-solution of \eqref{dysy}.

Consider the energy functional (see \cite{AttouchCzarnecki}) $${\cal E} : [t_0, +\infty) \rightarrow \R, \quad {\cal E}(t)=\frac12\|\dot x(t)\|^2+g(x(t))+\frac12\e({t})\|x(t)\|^2.$$
By using \eqref{dysy} we get
$$\frac{d}{dt} {\cal E}(t)=-\frac{\a}{t}\|\dot{x}(t)\|^2-\b\<\n^2 g(x(t))\dot{x}(t),\dot{x}(t)\>+\frac12\dot{\e}(t)\|x(t)\|^2,$$
and, since $\e$ is nonincreasing and $\n^2 g(x(t))$ is positive semidefinite, we obtain that 
$$\frac{d}{dt} {\cal E}(t) \le 0 \quad \forall t \geq t_0.$$
Consequently, ${\cal E}$ is nonincreasing, hence
$$\frac12\|\dot x(t)\|^2+g(x(t))+\frac12\e({t})\|x(t)\|^2\le \frac12\|\dot x(t_0)\|^2+g(x(t_0))+\frac12\e(t_0)\|x(t_0)\|^2 \ \quad \forall t \geq t_0.$$
From the fact that $g$ is bounded from below we obtain that $\dot{x}$ is bounded on $[t_0, T_{\max})$. Let $\|\dot{x}\|_\infty :=\sup_{t\in [t_0, T_{\max})}\|\dot{x}(t)\| < +\infty.$

Since $\|x(t)-x(t')\|\le\|\dot{x}\|_\infty |t-t'|$ for all $t,t' \in [t_0, T_{\max})$, there exists $\lim_{t\To T_{\max}}x(t)$, which shows that  $x$ is bounded on $[t_0, T_{\max})$. Since $\dot{x}(t)+\b\n g(x(t))=y(t)$ for all $t \in [t_0, T_{\max})$ and $\nabla g$ is Lipschitz continuous on bounded sets, it yields that $y$ is also bounded on  $[t_0,T_{\max})$. Hence $\lim_{t \rightarrow T_{\max}} \|x(t)\| + \|y(t)\|$ cannot be $+\infty$, thus $T_{\max}=+\infty,$ which completes the proof.
\end{proof}

\section{Asymptotic analysis}\label{sec3}

In this section we will show to which extent different assumptions we impose to the Tikhonov parametrization  $t \mapsto \e(t)$ influence the asymptotic behaviour of the trajectory $x$ generated by the dynamical system \eqref{dysy}. In particular, we are looking at the convergence of the function $g$ along the trajectory and the weak convergence of the trajectory.

We recall that the asymptotic analysis of the system \eqref{dysy} is carried out in the framework of the general assumptions stated in the introduction.

We start with a result which provides a setting that guarantees the convergence of $g(x(t))$ to $\min g$ as $t \rightarrow +\infty$.

\begin{theorem}\label{conv} Let $x$ be the unique global $C^2$-solution of \eqref{dysy}. Assume that
one of the following conditions is fulfilled:
\begin{itemize}
\item[{\rm (a)} ] $\int_{t_0}^{+\infty}\frac{\e(t)}{t}dt<+\infty$ and there exist $a > 1$ and $t_1\ge t_0$ such that
$$\dot{\e}(t)\le-\frac{a\b}{2}\e^2(t) \mbox{ for every }t\ge t_1;$$
\item[{\rm (b)}] there exists $a >0$ and $t_1\ge t_0$ such that
$$\e(t)\le\frac{a}{t} \mbox{ for every }t\ge t_1.$$
\end{itemize}
If $\a\ge 3$, then
$$\lim_{t \rightarrow + \infty}g(x(t))=\min g.$$
\end{theorem}
\begin{proof} 
Let be $x^* \in \argmin g$ and $2\le b\le \a-1$ be fixed. We introduce the following energy functional $\mathcal{E}_{b} : [t_0, +\infty) \rightarrow \R$,
\begin{align}\label{energy0}
\mathcal{E}_{b}(t)=& \ (t^2-\b(b+2-\a)t)\left(g(x(t))-\min g\right)+\frac{t^2\e(t)}{2}\|x(t)\|^2\nonumber\\
&\ +\frac12\|b(x(t)-x^*)+t(\dot{x}(t)+\b\n g(x(t)))\|^2+\frac{b(\a-1-b)}{2}\|x(t)-x^*\|^2.
\end{align}
For every $t \geq t_0$ it holds
\begin{align}\label{temp1}
\dot{\mathcal{E}_{b}}(t) = & \ (2t-\b(b+2-\a))\left(g(x(t))-\min g\right) \nonumber \\
&\ +(t^2-\b(b+2-\a)t)\<\n g(x(t),\dot{x}(t)\>+\frac{t^2\dot{\e}(t)+2t\e(t)}{2}\|x(t)\|^2+t^2\e(t)\<\dot{x}(t),x(t)\>\nonumber\\
&\ +\<(b+1)\dot{x}(t)+\b\n g(x(t))+t(\ddot{x}(t)+\b\n^2g(x(t))\dot{x}(t)), b(x(t)-x^*)+t(\dot{x}(t)+\b\n g(x(t)))\>\nonumber\\
&\ +b(\a-1-b)\<\dot{x}(t),x(t)-x^*\>.
\end{align}
Now, by using \eqref{dysy}, we get for every $t \geq t_0$
\begin{align}\label{temp2}
&\<(b+1)\dot{x}(t)+\b\n g(x(t))+t(\ddot{x}(t)+\b\n^2g(x(t))\dot{x}(t)), b(x(t)-x^*)+t(\dot{x}(t)+\b\n g(x(t)))\> \nonumber\\
= & \ \<(b+1-\a)\dot{x}(t)+(\b-t)\n g(x(t))-t\e(t)x(t), b(x(t)-x^*)+t(\dot{x}(t)+\b\n g(x(t)))\>\nonumber\\
= & \ b(b+1-\a)\<\dot{x}(t),x(t)-x^*\>+(b+1-\a)t\|\dot{x}(t)\|^2+(-t^2+\b (b+2-\a)t\<\dot{x}(t),\n g(x(t))\>\nonumber\\
&+(\b^2 t-\b t^2)\|\n g(x(t))\|^2-\e(t)t^2\<\dot{x}(t),x(t)\>-\b\e(t)t^2\<\n g(x(t)),x(t)\>\nonumber\\
&-bt\left\<\left(1-\frac{\b}{t}\right)\n g(x(t))+\e(t)x(t), x(t)-x^*\right\>. 
\end{align}

Let be $t_0':=\max(\b,t_0)$. For all $t \geq t_0'$  the function $g_t : \mathcal{H} \rightarrow \R, g_t(x)= \left(1-\frac{\b}{t}\right)g(x)+\frac{\e(t)}{2}\|x\|^2,$ is strongly convex, thus, one has
$$g_t(y)-g_t(x)\ge \<\n g_t(x),y-x\>+\frac{\e(t)}{2}\|y-x\|^2 \ \forall x,y\in\mathcal{H}.$$
By taking $x:=x(t)$ and $y:=x^*$ we get for every $t\ge t_0'$ 
\begin{align}\label{temp3}
-bt\left\<\left(1-\frac{\b}{t}\right)\n g(x(t))+\e(t)x(t), x(t)-x^*\right\>\le& -bt\left(1-\frac{\b}{t}\right)(g(x(t))-\min g) -bt\frac{\e(t)}{2}\|x(t)\|^2 \nonumber\\
&-bt\frac{\e(t)}{2}\|x(t)-x^*\|^2 +bt\frac{\e(t)}{2}\|x^*\|^2.
\end{align}

From \eqref{temp1}, \eqref{temp2} and \eqref{temp3} it follows that for every $t\ge t_0'$ it holds
\begin{align}\label{deriv}
\dot{\mathcal{E}_{b}}(t) \le & \ \big((2-b)t-\b(2-\a)\big)\left(g(x(t))-\min g\right)+bt\frac{\e(t)}{2}\|x^*\|^2 \nonumber \\
& \ +\left(t^2\frac{\dot{\e}(t)}{2}+(2-b)t\frac{\e(t)}{2}\right)\|x(t)\|^2-bt\frac{\e(t)}{2}\|x(t)-x^*\|^2\nonumber\\
& \ +(b+1-\a)t\|\dot{x}(t)\|^2+(\b^2 t-\b t^2)\|\n g(x(t))\|^2 -\b\e(t)t^2\<\n g(x(t)),x(t)\>.
\end{align}
At this point we treat the situations $\a >3$ and $\a=3$ separately.\vspace{1ex}

{\bf The case $\alpha > 3$ and $2<b< \a-1$.} We will carry out the analysis by addressing the settings provided by the conditions (a) and (b) separately.

{\it Condition {\rm (a)} holds:} Assuming that condition {\rm (a)} holds, there exist $a >1$ and $t_1\ge t_0'$ such that
$$\dot{\e}(t)\le-\frac{a\b}{2}\e^2(t) \ \mbox{ for every }t\ge t_1.$$
Using that
\begin{align}\label{temp5}
-\b\e(t)t^2\<\n g(x(t)),x(t)\> \leq  \frac{\b t^2}{a}\|\n g(x(t))\|^2+\frac{a\b\e^2(t)t^2}{4}\|x(t)\|^2,
\end{align}
\eqref{deriv} leads to the following estimate
\begin{align}\label{deriv1}
\dot{\mathcal{E}_{b}}(t) \le & \ \big((2-b)t-\b(2-\a)\big)\left(g(x(t))-\min g\right)+bt\frac{\e(t)}{2}\|x^*\|^2 \nonumber \\
& \ +\left(t^2\frac{\dot{\e}(t)}{2}+(2-b)t\frac{\e(t)}{2}+\frac{a\b\e^2(t)t^2}{4}\right)\|x(t)\|^2-bt\frac{\e(t)}{2}\|x(t)-x^*\|^2\nonumber\\
& \ +(b+1-\a)t\|\dot{x}(t)\|^2+\left(\b^2 t-\b\left(1-\frac{1}{a}\right) t^2\right)\|\n g(x(t))\|^2,
\end{align}
which holds for every $t\ge t_1.$

Since $a> 1$ and $b > 2$, we notice that  for every $t\ge t_1$ it holds
$$t^2\frac{\dot{\e}(t)}{2}+(2-b)t\frac{\e(t)}{2}+\frac{a\b\e^2(t)t^2}{4}\le 0.$$
On the other hand, we have that
$$\b^2 t-\b\left(1-\frac{1}{a}\right) t^2\le -\b\frac{a-1}{2a}t^2\mbox{ for every }t\ge \frac{2a\b}{a-1}$$
and
$$(2-b)t-\b(2-\a)\le 0\mbox{ for every } t\ge\frac{\b(\a-2)}{b-2}.$$

We define $t_2:=\max\left(t_1,\frac{2a\b}{a-1},\frac{\b(\a-2)}{b-2}\right)$. According to \eqref{deriv1}, it holds for every $t\ge t_2$
\begin{align}\label{deriv2}
& \ \dot{\mathcal{E}_{b}}(t) - \big((2-b)t-\b(2-\a)\big)\left(g(x(t))-\min g\right)-\left(t^2\frac{\dot{\e}(t)}{2}+(2-b)t\frac{\e(t)}{2}+\frac{a\b\e^2(t)t^2}{4}\right)\|x(t)\|^2 \nonumber \\
& \ +bt\frac{\e(t)}{2}\|x(t)-x^*\|^2+(\a-1-b)t\|\dot{x}(t)\|^2+\b\frac{a-1}{2a}t^2\|\n g(x(t))\|^2\nonumber\\
\le&  \ bt\frac{\e(t)}{2}\|x^*\|^2.
\end{align}

{\it Condition {\rm (b)} holds:} Assuming now that condition {\rm (b)} holds, there exist $a >0$ and $t_1\ge t_0'$ such that
$$\e(t)\le\frac{a}{t} \ \mbox{ for every }t\ge t_1.$$
Further, the monotonicity of $\n g$ and the fact that $\n g(x^*)=0$ implies that
$$\<\n g(x(t)),x(t)-x^*\>\ge 0 \mbox{ for every }t\ge t_0.$$
Using that
\begin{align}\label{temp5b}
-\b\e(t)t^2\<\n g(x(t)),x(t)\>\le -\b\e(t)t^2\<\n g(x(t)),x^*\> \leq \frac{\b t^3\e(t)}{2a}\|\n g(x(t))\|^2+\frac{a\b\e(t)t}{2}\|x^*\|^2,
\end{align}
\eqref{deriv} leads to the following estimate
\begin{align}\label{deriv1b}
\dot{\mathcal{E}_{b}}(t) \le & \ \big((2-b)t-\b(2-\a)\big)\left(g(x(t))-\min g\right)+(b+a\b)t\frac{\e(t)}{2}\|x^*\|^2 \nonumber \\
& \ +\left(t^2\frac{\dot{\e}(t)}{2}+(2-b)t\frac{\e(t)}{2}\right)\|x(t)\|^2-bt\frac{\e(t)}{2}\|x(t)-x^*\|^2\nonumber\\
& \ +(b+1-\a)t\|\dot{x}(t)\|^2+\left(\b^2 t-\b t^2+\frac{\b t^3\e(t)}{2a}\right)\|\n g(x(t))\|^2
\end{align}
for every $t\ge t_1$.

Since $b > 2$, we have that for every $t\ge t_1$ it holds
$$t^2\frac{\dot{\e}(t)}{2}+(2-b)t\frac{\e(t)}{2}\le 0.$$
On the other hand, since
$$-\b t^2+\frac{\b t^3\e(t)}{2a}\le -\frac{\b}{2}t^2$$ holds for every $t\ge t_1$, it follows that
\begin{equation}\label{usi}\b^2 t-\b t^2+\frac{\b t^3\e(t)}{2a}\le -\frac{\b}{4}t^2 \ \mbox{ for every }t\ge \max(t_1,4\b).
\end{equation}
We recall that
$$(2-b)t-\b(2-\a)\le 0\mbox{ for every } t\ge\frac{\b(\a-2)}{b-2}.$$
We define $t_2:=\max\left(t_1,4\b,\frac{\b(\a-2)}{b-2}\right).$ According to \eqref{deriv1b}, it holds for every $t \geq t_2$
\begin{align}\label{deriv2b}
& \  \dot{\mathcal{E}_{b}}(t) - ((2-b)t-\b(2-\a))\left(g(x(t))-\min g\right)-\left(t^2\frac{\dot{\e}(t)}{2}+(2-b)t\frac{\e(t)}{2}\right)\|x(t)\|^2 \nonumber \\
&\ +bt\frac{\e(t)}{2}\|x(t)-x^*\|^2+(\a-1-b)t\|\dot{x}(t)\|^2+\frac{\b}{4}t^2\|\n g(x(t))\|^2\nonumber\\
\le & \ (b+a\b)t\frac{\e(t)}{2}\|x^*\|^2.
\end{align}

From now on we will treat the two cases together. According to \eqref{deriv2}, in case {\rm (a)}, and to \eqref{deriv2b}, in case {\rm (b)}, we  obtain
$$\dot{\mathcal{E}_{b}}(t)\le lt\frac{\e(t)}{2}\|x^*\|^2$$
for every $t\ge t_2$, where $l:=b\mbox{ and } t_2=\max\left(t_1,\frac{2a\b}{a-1},\frac{\b(\a-2)}{b-2}\right)$, in case {\rm (a)}, and $l:=b+a\b\mbox{ and }t_2=\max\left(t_1,4\b,\frac{\b(\a-2)}{b-2}\right)$ in case {\rm (b)}.

By integrating the latter inequality on the interval $[t_2,T]$, where $T \geq t_2$ is arbitrarily chosen, we obtain
$$\mathcal{E}_{b}(T)\le \mathcal{E}_{b}(t_2)+\frac{l\|x^*\|^2}{2}\int_{t_2}^T t\e(t)dt.$$
On the other hand,
$$\mathcal{E}_{b}(t)\ge (t^2-\b(b+2-\a)t)\left(g(x(T))-\min g\right) \ \forall t \geq t_0,$$ hence, for every $T\ge\max( \b(b+2-\a),t_3)$ we get 
$$0 \leq g(x(T))-\min g \le \frac{\mathcal{E}_{b}(t_2)}{T^2-\b(b+2-\a)T}+\frac{l\|x^*\|^2}{2}\frac{1}{T^2-\b(b+2-\a)T}\int_{t_2}^T t\e(t)dt.$$
Obviously,
$$\lim_{T\To+\infty}\frac{\mathcal{E}_{b}(t_3)}{T^2-\b(b+2-\a)T}=0.$$
Further,  Lemma \ref{nullimit} applied to the functions $\varphi(t)=t^2$ and $f(t)=\frac{\e(t)}{t}$ provides
$$\lim_{T\To+\infty}\frac{1}{T^2}\int_{t_2}^T t^2\frac{\e(t)}{t}dt=0,$$
hence,
$$\lim_{T\To+\infty}\frac{1}{T^2-\b(b+2-\a)T}\int_{t_2}^T t\e(t)dt=0$$
and, consequently,
$$\lim_{T\To+\infty}g(x(T))=\min g.$$

{\bf The case $\a=3$ and $b=2$.} In this case the energy functional reads
$$\mathcal{E}_{2}(t)=(t^2-\b t)\left(g(x(t))-\min g\right)+\frac{t^2\e(t)}{2}\|x(t)\|^2+\frac12\|2(x(t)-x^*)+t(\dot{x}(t)+\b\n g(x(t)))\|^2$$
 for every $ t \geq t_0$. We will address again the settings provided by the conditions {\rm (a)} and {\rm (b)} separately.\vspace{1ex}

{\it Condition {\rm (a)} holds:} Relation \eqref{deriv1} becomes
\begin{align}\label{deriv00}
\dot{\mathcal{E}_{2}}(t) \le & \ \b\left(g(x(t))-\min g\right)+t\e(t)\|x^*\|^2+\left(t^2\frac{\dot{\e}(t)}{2}+\frac{a\b\e^2(t)t^2}{4}\right)\|x(t)\|^2-t\e(t)\|x(t)-x^*\|^2\nonumber\\
&\ +\left(\b^2 t-\b\left(1-\frac{1}{a}\right) t^2\right)\|\n g(x(t))\|^2\nonumber
\end{align}
for every $t \geq t_1$. Consequently, for $t_3:= \max\left(t_1,\frac{\b a}{a-1}\right)$, we have
\begin{equation}\label{3a}\dot{\mathcal{E}_{2}}(t)\le \b\left(g(x(t))-g^*\right)+t\e(t)\|x^*\|^2
\end{equation}
for every $t \geq t_3$. After multiplication with $(t-\b)$, it yields
$$
t(t-\b)\dot{\mathcal{E}_{2}}(t)\le  \b t(t-\b)\left(g(x(t))-g^*\right)+t^2(t-\b)\e(t)\|x^*\|^2\le \b\mathcal{E}_{2}(t)+t^2(t-\b)\e(t)\|x^*\|^2$$
for every $t\ge t_3$. Dividing by $(t-\b)^2$ we obtain 
$$\frac{t}{t-\b}\dot{\mathcal{E}_{2}}(t)\le \frac{\b}{(t-\b)^2}\mathcal{E}_{2}(t)+\frac{t^2}{t-\b}\e(t)\|x^*\|^2$$
or, equivalently,
\begin{equation}\label{deri00}
\frac{d}{dt}\left(\frac{t}{t-\b}\mathcal{E}_{2}(t)\right)\le \frac{t^2}{t-\b}\e(t)\|x^*\|^2 \mbox{ for every } t\ge t_3.
\end{equation}

{\it Condition {\rm (b)} holds:} We define  $t_3:=\max\left(t_1,4\b\right)$. Relation \eqref{deriv1b} becomes
\begin{equation}\label{3b}\dot{\mathcal{E}_{2}}(t)\le \b\left(g(x(t))-g^*\right)+\frac{2+a\b}{2}t\e(t)\|x^*\|^2,
\end{equation}
for every $t\ge t_3$. Repeating the above steps for the inequality \eqref{3b} we obtain
\begin{equation}\label{deri00b}
\frac{d}{dt}\left(\frac{t}{t-\b}\mathcal{E}_{2}(t)\right)\le \frac{2+a_1\b}{2}\frac{t^2}{t-\b}\e(t)\|x^*\|^2 \mbox{ for every } t\ge t_3.
\end{equation}

From now on we will treat the two cases together. According to \eqref{deri00}, in case {\rm (a)}, and to \eqref{deri00b}, in case {\rm (b)}, we  obtain
$$\frac{d}{dt}\left(\frac{t}{t-\b}\mathcal{E}_{2}(t)\right)\le l\frac{t^2}{t-\b}\e(t)\|x^*\|^2$$
for every $t\ge t_3$, where $l:=1\mbox{ and } t_3=\max\left(t_1,\frac{\b(\a-1)}{b-2}\right)$, in case {\rm (a)}, and $l:=\frac{2+a\b}{2}\mbox{ and }t_3=\max(t_1,4\b)$ in case {\rm (b)}.

By integrating the latter inequality on an interval $[t_3,T]$, where $T \geq t_3$ is arbitrarily chosen, we obtain
$$\frac{T}{T-\b}\mathcal{E}_{2}(T)\le \frac{t_3}{t_3-\b}\mathcal{E}_{2}(t_3)+l\|x^*\|^2\int_{t_3}^T \frac{t^2}{t-\b}\e(t) dt.$$
On the other hand,
$$\mathcal{E}_{2}(t)\ge (t^2-\b t)\left(g(x(t))-\min g\right)$$ for every $t \ge t_0$, hence, for every
$T \ge \max( \b,t_3)=t_3$ we get
$$0 \leq g(x(T))-\min g \le \frac{1}{T^2} \frac{t_3}{t_3-\b}\mathcal{E}_{2}(t_3)+l\|x^*\|^2\frac{1}{T^2}\int_{t_3}^T \frac{t^2}{t-\b}\e(t)dt.$$
Obviously,
$$\lim_{T\To+\infty}\frac{1}{T^2} \frac{t_3}{t_3-\b}\mathcal{E}_{2}(t_3)=0.$$
Lemma \ref{nullimit}, applied this time to the functions $\varphi(t)=\frac{t^3}{t-\b}$ and $f(t)=\frac{\e(t)}{t}$, yields
$$\lim_{T\To+\infty} \frac{T-\b}{T^3}\int_{t_3}^T \frac{t^3}{t-\b}\frac{\e(t)}{t}dt=0.$$
Consequently,
$$\lim_{T\To+\infty}\frac{1}{T^2}\int_{t_3}^T \frac{t^2}{t-\b}\e(t)dt=0,$$
hence
$$\lim_{T\To+\infty}g(x(T))=\min g.$$
\end{proof}

\begin{remark}\label{rema}
One can easily notice that, in case $\beta >0$, the fact that there exist $a >1$ and $t_1 \geq t_0$ such that $\dot{\e}(t)\le-\frac{a\b}{2}\e^2(t)$ for every $t\ge t_1$ implies that $\int_{t_0}^{+\infty}\frac{\e(t)}{t}dt<+\infty$.
\end{remark}

The next theorem shows that, by strengthening the integrability condition $\int_{t_0}^{+\infty}\frac{\e(t)}{t}dt<+\infty$ (which is actually required in both settings {\rm (a)} and {\rm (b)} of Theorem \ref{conv}),  a rate of $\mathcal{O}(1/t^2)$ ca be guaranteed for the convergence of $g(x(t))$ to $\min g$.

\begin{theorem}\label{rates} Let $x$ be the unique global $C^2$-solution of \eqref{dysy}. Assume that
$$\int_{t_0}^{+\infty}t\e(t)dt<+\infty$$
and that one of the following conditions is fulfilled:
\begin{itemize}
\item[{\rm (a)}] there exist $a>1$ and $t_1\ge t_0$ such that
$$\dot{\e}(t)\le-\frac{a\b}{2}\e^2(t)\mbox{ for every }t\ge t_1;$$
\item[{\rm (b)}] there exist $a>0$ and $t_1\ge t_0$ such that
$$\e(t)\le\frac{a}{t}\mbox{ for every }t\ge t_1.$$
\end{itemize}
If $\a\ge 3$, then
$$g(x(t))-\min g =  \mathcal{O}\left(\frac{1}{t^2}\right).$$
In addition, if $\a>3$, then the trajectory $x$ is bounded and
$$t\left(g(x(t))-\min g\right),\,t\|\dot{x}(t)\|^2,\,t\e(t)\|x(t)-x^*\|^2,\,t\e(t)\|x(t)\|^2,\,t^2\|\n g(x(t))\|^2\in L^1([t_0,+\infty),\R)$$
for every arbitrary $x^* \in \argmin g$.
\end{theorem}

\begin{proof}
Let be $x^* \argmin g$ and $2 \leq b \leq \alpha-1$ fixed. We will use the energy functional introduced in the proof of the previous theorem and some of the estimate we derived for it. We will treat again the situations $\a >3$ and $\a=3$ separately.

{\bf The case $\alpha > 3$ and $2<b< \a-1$.} As we already noticed in the proof of Theorem \ref{conv}, according to \eqref{deriv2}, in case {\rm (a)}, and to \eqref{deriv2b}, in case {\rm (b)}, we have
$$\dot{\mathcal{E}_{b}}(t)\le lt\frac{\e(t)}{2}\|x^*\|^2 \ \mbox{for every} \ t \ge t_2,$$
where $l:=b\mbox{ and } t_2=\max\left(t_1,\frac{2a\b}{a-1},\frac{\b(\a-2)}{b-2}\right)$, in case {\rm (a)}, and $l:=b+a\b\mbox{ and }t_2=\max\left(t_1,4\b,\frac{\b(\a-2)}{b-2}\right)$ in case {\rm (b)}.

Using that $t\e(t)\in L^1([t_0,+\infty),\R)$ and that $t \mapsto \mathcal{E}_{b}(t)$ is bounded from below, from Lemma \ref{fejer-cont1} it follows that the limit $\lim_{t\To+\infty}\mathcal{E}_{b}(t)$ exists.
Consequently, $t \mapsto \mathcal{E}_{b}(t)$ is bounded, which implies that there exist $K>0$ and $t'\ge t_0$  such that
$$0 \leq g(x(t))-\min g\le \frac{K}{t^2} \mbox{ for every }t\ge t'.$$
In addition, the function $t \mapsto \|x(t)-x^*\|^2$ is bounded, hence the trajectory $x$ is bounded. Since $t \mapsto \|b(x(t)-x^*)+t(\dot{x}(t)+\b\n g(x(t)))\|^2$ is also bounded, the inequality
$$\|t(\dot{x}(t)+\b\n g(x(t)))\|^2 \le 2\|b(x(t)-x^*)+t(\dot{x}(t)+\b\n g(x(t)))\|^2+2b^2\|x(t)-x^*\|^2,$$
which is true for every $t \geq t_0$, leads to
$$\|\dot{x}(t)+\b\n g(x(t))\| = \mathcal{O}\left(\frac{1}{t}\right).$$

By integrating relation \eqref{deriv2}, in case {\rm (a)}, and relation \eqref{deriv2b}, in case {\rm (b)}, on an interval $[t_2, s]$, where $s \ge t_3$ is arbitrarily chosen, and by letting afterwards $s$ converge to $+\infty$, we obtain
$$t\left(g(x(t))-\min g \right),\,t\|\dot{x}(t)\|^2,\,t\e(t)\|x(t)-x^*\|^2, t^2\|\n g(x(t))\|^2 \in L^1([t_0,+\infty),\R).$$
The boundedness of the trajectory and the condition on the Tikhonov parametrization guarantee that
$$t\e(t)\|x(t)\|^2\in L^1([t_0,+\infty),\R).$$

{\bf The case $\alpha = 3$ and $b=2$.} As we already noticed in the proof of Theorem \ref{conv}, according to \eqref{deri00}, in case {\rm (a)}, and to \eqref{deri00b}, in case {\rm (b)}, we  obtain
$$\frac{d}{dt}\left(\frac{t}{t-\b}\mathcal{E}_{2}(t)\right)\le l\frac{t^2}{t-\b}\e(t)\|x^*\|^2 \ \mbox{for every} \ t \geq t_3,$$
where $l=1\mbox{ and } t_3=\max\left(t_1,\frac{\b(\a-1)}{b-2}\right)$, in case {\rm (a)}, and $l=\frac{2+a\b}{2}\mbox{ and }t_3=\max(t_1,4\b)$ in case {\rm (b)}.

Since $t\e(t)\in L^1([t_0,+\infty),\R)$ and $\e(t)$ is nonnegative, obviously $\frac{t^2}{t-\b}\e(t)\|x^*\|^2\in L^1([t_2,+\infty),\R)$. Using that $t \mapsto \frac{t}{t-\b}\mathcal{E}_{2}(t)$ is bounded from below, from  Lemma \ref{fejer-cont1} it follows that the limit $\lim_{t\To+\infty}\frac{t}{t-\b}\mathcal{E}_{2}(t)$ exists. Consequently, the limit $\lim_{t\To+\infty}\mathcal{E}_{2}(t)$ also exists and $t \mapsto \mathcal{E}_{2}(t)$ is bounded. This implies that there exist $K>0$ and $t'\ge t_0$  such that
$$0 \leq g(x(t))-\min g\le \frac{K}{t^2} \mbox{ for every }t\ge t'.$$
\end{proof}

The next result shows that the statements of Theorem \ref{rates} can be strengthened in case $\alpha >3$. 
\begin{theorem}\label{limits}
Let $x$ be the unique global $C^2$-solution of \eqref{dysy}. Assume that
$$\int_{t_0}^{+\infty}t\e(t)dt<+\infty$$
and that one of the following conditions is fulfilled:
\begin{itemize}
\item[{\rm (a)}] there exist $a>1$ and $t_1\ge t_0$ such that
$$\dot{\e}(t)\le-\frac{a\b}{2}\e^2(t)\mbox{ for every }t\ge t_1;$$
\item[{\rm (b)}] there exist $a>0$ and $t_1\ge t_0$ such that
$$\e(t)\le\frac{a}{t}\mbox{ for every }t\ge t_1.$$
\end{itemize}
Let be an arbitrary $x^* \in \argmin g$. If $\a> 3$, then
 $$t\<\n g(x(t)),x(t)-x^*\>\in L^1([t_0,+\infty),\R)$$
and the limits
 $$\lim_{t\To+\infty}\|x(t)-x^*\|\in\R \ \mbox{and} \ \lim_{t\To+\infty}t\<\dot{x}(t)+\b\n g(x(t)),x(t)-x^*\>\in\R$$
exist. In addition,
$$g(x(t))-\min g=o\left(\frac{1}{t^2}\right),\,\|\dot{x}(t)+\b\n g(x(t))\|=o\left(\frac{1}{t}\right) \ \mbox{and} \ \lim_{t\To+\infty}t^2\e(t)\|x(t)\|^2=0.$$
\end{theorem}
\begin{proof}
Since $\a >3$ we can choose $2<b< \a-1$.  From \eqref{temp1} and \eqref{temp2} we have that
\begin{align}\label{derivlim}
\dot{\mathcal{E}_{b}}(t) = & \ (2t-\b(b+2-\a))\left(g(x(t))-\min g\right)+\left(t^2\frac{\dot{\e}(t)}{2}+t\e(t)\right)\|x(t)\|^2\nonumber\\
& \ +(b+1-\a)t\|\dot{x}(t)\|^2+(\b^2 t-\b t^2)\|\n g(x(t))\|^2-\b\e(t)t^2\<\n g(x(t)),x(t)\>\nonumber\\
& \ -bt\left\<\left(1-\frac{\b}{t}\right)\n g(x(t))+\e(t)x(t), x(t)-x^*\right\>\mbox{ for every }t\ge t_0.
\end{align}
We will address the settings provided by the conditions  {\rm (a)} and {\rm (b)} separately.

{\it Condition {\rm (a)} holds:} In this case we estimate $-\b\e(t)t^2\<\n g(x(t)),x(t)\>$ just as in \eqref{temp5} and from \eqref{derivlim} we obtain 
\begin{align}\label{derivlim1}
\dot{\mathcal{E}_{b}}(t) \le& \ (2t-\b(b+2-\a))\left(g(x(t))-\min g\right)+\left(t^2\frac{\dot{\e}(t)}{2}+t\e(t)+\frac{a\b\e^2(t)t^2}{4}\right)\|x(t)\|^2 \nonumber\\
& \ +(b+1-\a)t\|\dot{x}(t)\|^2+\left(\b^2 t-\b\left(1-\frac{1}{a}\right) t^2\right)\|\n g(x(t))\|^2\nonumber\\
& \ -bt\left\<\left(1-\frac{\b}{t}\right)\n g(x(t))+\e(t)x(t), x(t)-x^*\right\> \mbox{ for every }t\ge t_0.
\end{align}

We define $t_2:=\max\left(\b,t_1,\frac{\b a}{a-1}\right).$ By using condition {\rm (a)}, neglecting the nonpositive terms and afterwards integrating on the interval $[t_2,t]$, with arbitrary $t \geq t_2$, we obtain
\begin{align}\label{temp6}
\int_{t_2}^t bs\left\<\left(1-\frac{\b}{s}\right)\n g(x(s)), x(s)-x^*\right\>\le & \ \mathcal{E}_{b}(t_2)-\mathcal{E}_{b}(t)+\int_{t_2}^t (2s-\b(b+2-\a))\left(g(x(s))-\min g\right)ds \nonumber\\
&-\int_{t_2}^t bs\left(1-\frac{\b}{s}\right)\left\<\e(s)x(s), x(s)-x^*\right\>+\int_{t_2}^t s\e(s)\|x(s)\|^2ds.
\end{align}
For every $s\ge t_2$, by the monotonicity of $\n g$, we have 
$\left\<\n g(x(s)), x(s)-x^*\right\>\ge 0$. Further, it holds $$bs\left(1-\frac{\b}{s}\right)\e(s)\left |\left\<x(s), x(s)-x^*\right\> \right |\le \left(1-\frac{\b}{s}\right)\frac{bs\e(s)}{2}(\|x(s)\|^2+\|x(s)-x^*\|^2).$$
By letting in \eqref{temp6} $s$ converge to $+\infty$ and by taking into account that, according to Theorem \ref{rates}, 
$$t\e(t)\|x(t)\|^2,\,t\e(t)\|x(t)-x^*\|^2,\,(2t-\b(b+2-\a))\left(g(x(t))-g^*\right)\in L^1([t_0,+\infty),\R)$$ and also that $t \mapsto \mathcal{E}_{b}(t)$ is bounded, we obtain
\begin{equation}\label{temp7}
t\<\n g(x(t)),x(t)-x^*\>\in L^1([t_0,+\infty),\R).
\end{equation}

{\it Condition {\rm (b)} holds:} In this case we estimate $-\b\e(t)t^2\<\n g(x(t)),x(t)\>$ just as in \eqref{temp5b} and from \eqref{derivlim} we obtain
\begin{align}\label{derivlim1b}
\dot{\mathcal{E}_{b}}(t) \le & \ (2t-\b(b+2-\a))\left(g(x(t))-\min g\right)+\left(t^2\frac{\dot{\e}(t)}{2}+t\e(t)\right)\|x(t)\|^2\nonumber\\
& \ +(b+1-\a)t\|\dot{x}(t)\|^2+\left(\b^2 t-\b t^2+\frac{\b \e(t)t^3}{2a}\right)\|\n g(x(t))\|^2+\frac{a_1\b \e(t)t}{2}\|x^*\|^2\nonumber\\
&\ -bt\left\<\left(1-\frac{\b}{t}\right)\n g(x(t))+\e(t)x(t), x(t)-x^*\right\> \mbox{ for every }t\ge t_0.
\end{align}
We define $t_2:=\max\left(4\b,t_1\right).$ According to \eqref{usi} we have that $\b^2 t-\b t^2+\frac{\b \e(t)t^3}{2a_1}\le 0$ for every $t\ge t_2$. By using condition {\rm (b)}, neglecting the nonpositive terms and afterwards integrating on the interval $[t_2,t]$, with arbitrary $t \geq t_2$, we obtain
\begin{align}\label{temp6b}
\int_{t_2}^t bs\left\<\left(1-\frac{\b}{s}\right)\n g(x(s)), x(s)-x^*\right\>\le & \ \mathcal{E}_{b}(t_2)-\mathcal{E}_{b}(t)+\int_{t_2}^t (2s-\b(b+2-\a))\left(g(x(s))-\min g\right)ds \nonumber \\
& \ -\int_{t_2}^t bs\left(1-\frac{\b}{s}\right)\left\<\e(s)x(s), x(s)-x^*\right\>+\int_{t_2}^t s\e(s)\|x(s)\|^2ds \nonumber\\ 
& \ +\frac{a\b}{2}\|x^*\|^2\int_{t_2}^t s\e(s)ds.
\end{align}
From here, by using the similar arguments as for the case {\rm (a)}, we obtain \eqref{temp7}.

Consider now, $b_1,b_2\in(2, \a-1),\,b_1\neq b_2.$ Then for every $t \geq t_0$ we have
$$\mathcal{E}_{b_1}(t)-\mathcal{E}_{b_2}(t)=(b_1-b_2)\left(-\b t(g(x(t))-\min g)+t\<\dot{x}(t)+\b\n g(x(t)),x(t)-x^*\>+\frac{\a-1}{2}\|x(t)-x^*\|^2\right).$$
According to Theorem \ref{rates}, the limits
$$\lim_{t\To+\infty}(\mathcal{E}_{b_1}(t)-\mathcal{E}_{b_2}(t)) \in \R \ \mbox{and} \ \lim_{t\To+\infty}t(g(x(t))-g^*) \in \R$$
exist, consequently, the limit
$$\lim_{t\To+\infty}\left(t\<\dot{x}(t)+\b\n g(x(t)),x(t)-x^*\>+\frac{\a-1}{2}\|x(t)-x^*\|^2\right)$$
also exists.
For every $t \geq t_0$ we define 
$$k(t)=t\<\dot{x}(t)+\b\n g(x(t)),x(t)-x^*\>+\frac{\a-1}{2}\|x(t)-x^*\|^2$$ 
and
$$q(t)=\frac12\|x(t)-x^*\|^2+\b\int_{t_0}^t\<\n g(x(s)),x(s)-x^*\>ds.$$
Then
$$(\a-1)q(t)+t\dot{q}(t)=k(t)+\b(\a-1)\int_{t_0}^t\<\n g(x(s)),x(s)-x^*\>ds \ \mbox{for every} \ t \geq t_0.$$
From \eqref{temp7} and the fact that $k(t)$ has a limit whenever $t\To+\infty$, we obtain that
$(\a-1)q(t)+t\dot{q}(t)$ has a limit when $t\To+\infty$. According to Lemma \ref{lemmader}, $q(t)$ has a limit when $t\To+\infty.$ By using \eqref{temp7} again we obtain that the limit
$$\lim_{t\To+\infty}\|x(t)-x^*\| \in \R$$
exists and, consequently, the limit
$$\lim_{t\To+\infty}t\<\dot{x}(t)+\b\n g(x(t)),x(t)-x^*\> \in \R$$
also exists. On the other hand, we notice that for every $t \geq t_0$ the energy functional can be written as
\begin{align}\label{energy1}
\mathcal{E}_{b}(t)=& \ (t^2-\b(b+2-\a)t)\left(g(x(t))-\min g\right)+\frac{t^2\e(t)}{2}\|x(t)\|^2 \nonumber\\
& \ +\frac{t^2}{2}\|\dot{x}(t)+\b\n g(x(t))\|^2+bt\<\dot{x}(t)+\b\n g(x(t)),x(t)-x^*\>+\frac{b(\a-1)}{2}\|x(t)-x^*\|^2.
\end{align}
Since the limits
$$\lim_{t\To+\infty}\mathcal{E}_{b}(t) \in \R \ \mbox{and} \ \lim_{t\To+\infty}\left(bt\<\dot{x}(t)+\b\n g(x(t)),x(t)-x^*\>+\frac{b(\a-1)}{2}\|x(t)-x^*\|^2\right) \in \R$$
exist, it follows that the limit
$$\lim_{t\To+\infty}\left((t^2-\b(b+2-\a)t)\left(g(x(t))-\min g\right)+\frac{t^2\e(t)}{2}\|x(t)\|^2+\frac{t^2}{2}\|\dot{x}(t)+\b\n g(x(t))\|^2\right) \in \R$$
exists, too.

We define
 $$\varphi:[t_0,+\infty)\To\R, \ \varphi(t)=(t^2-\b(b+2-\a)t)\left(g(x(t))-g^*\right)+\frac{t^2\e(t)}{2}\|x(t)\|^2+\frac{t^2}{2}\|\dot{x}(t)+\b\n g(x(t))\|^2,$$
and notice that for sufficiently large $t$ it holds
 $$0\le\frac{\varphi(t)}{t}\le 2t\left(g(x(t))-\min g\right)+\frac{t\e(t)}{2}\|x(t)\|^2+\frac{t}{2}\|\dot{x}(t)+\b\n g(x(t))\|^2.$$
According to Theorem \ref{rates} the right hand side of the above inequality is of class $L^1([t_0,+\infty),\R).$

Hence, $$\frac{\varphi(t)}{t}\in L^1([t_0,+\infty),\R).$$
Since $\frac{1}{t}\not\in L^1([t_0,+\infty),\R)$ and the limit $\lim_{t \To +\infty} \varphi(t) \in \R$ exists,
it must hold that $\lim_{t\To+\infty}\varphi(t)=0.$ Consequently,
$$\lim_{t\To+\infty}(t^2-\b(b+2-\a)t)\left(g(x(t))-\min g\right)=\lim_{t\To+\infty}\frac{t^2\e(t)}{2}\|x(t)\|^2=\lim_{t\To+\infty}\frac{t^2}{2}\|\dot{x}(t)+\b\n g(x(t))\|^2=0$$
and the proof is complete.
\end{proof}

Working in the hypotheses of Theorem \ref{limits} we can prove also the weak convergence of the trajectories generated by \eqref{dysy} to a minimizer of the objective function $g$. 

\begin{theorem}\label{convergencetraj}
Let $x$ be the unique global $C^2$-solution of \eqref{dysy}. Assume that
$$\int_{t_0}^{+\infty}t\e(t)dt<+\infty$$
and that one of the following conditions is fulfilled:
\begin{itemize}
\item[{\rm (a)}] there exist $a>1$ and $t_1\ge t_0$ such that
$$\dot{\e}(t)\le-\frac{a\b}{2}\e^2(t)\mbox{ for every }t\ge t_1;$$
\item[{\rm (b)}] there exist $a>0$ and $t_1\ge t_0$ such that
$$\e(t)\le\frac{a}{t}\mbox{ for every }t\ge t_1.$$
\end{itemize}
If $\a> 3$, then $x(t)$ converges weakly to an element in $\argmin g$ as $t \To +\infty$.
\end{theorem}
\begin{proof} We will to apply the continuous version of the Opial Lemma (Lemma \ref{Opial}) for $S=\argmin g.$ According to Theorem \ref{limits}, the limit
 $$\lim_{t\To+\infty}\|x(t)-x^*\| \in \R$$ 
exists for every $x^*\in\argmin g$ .

Further, let $\ol x\in \mathcal{H}$ be a weak sequential limit point of  $x(t)$. This means that there exists a sequence $(t_n)_{n \in \N} \subseteq [t_0, +\infty)$ such that $\lim_{n \To \infty} t_n = +\infty$ and $x(t_n)$ converges weakly to $\ol x$ as $n \To \infty$. Since $g$ is weakly lower semicontinuous,  we have that
$$g(\ol x)\le\liminf_{n\To+\infty}g(x(t_n)).$$
On the other hand, according to Theorem \ref{rates},
 $$\lim_{t\To +\infty}g(x(t))=\min g,$$
consequently one has $g(\ol x)\le \min g$, which shows that $\ol x\in\argmin g.$

The convergence of the trajectory is a consequence of Lemma \ref{Opial}.
\end{proof}

\begin{remark} \label{rem1}
We proved in this section that the 
convergence rate of $o\left(\frac{1}{t^2}\right)$ for $g(x(t))$, the converge rate of $o\left(\frac{1}{t}\right)$ for $\|\dot{x}(t)+\b\n g(x(t))\|$ and the weak convergence of the trajectory to a minimizer of $g$ that have been obtained in \cite{att-p-r-jde2016} for the dynamical system with Hessian driven damping \eqref{dysy-hess} are preserved when this system is enhanced with a Tikhonov regularization term. In addition, in the case when the Hessian driven damping term is removed, which is the case when $\beta =0$, we recover the results provided in \cite{ACR} for the dynamical system \eqref{dysy-sbc-tikh} with Tikhonov regularization term. In this setting, we have to assume in Theorem \ref{conv} just that $\int_{t_0}^{+\infty}\frac{\e(t)}{t}dt<+\infty$, and in the theorems \ref{rates} -  \ref{convergencetraj} just that $\int_{t_0}^{+\infty}t\e(t)dt<+\infty$, since condition {\rm (a)} is automatically fulfilled.

\end{remark}

\section{Strong convergence to the minimum norm solution}\label{sec4}

In this section we will continue the investigations we did at the end of Section \ref{sec3}, by working in the same setting, on the behaviour of the trajectory of the dynamical system \eqref{dysy} by concentrating on strong convergence. In particular, we will provide conditions on the Tikhonov parametrization $t \mapsto \e(t)$ which will guarantee that the trajectory converges to a minimum norm solution of $g$, which is the element of minimum norm of the nonempty convex closed set $\argmin g$.
We start with the following result. 

\begin{lemma}\label{ergl}
Let $x$ be the unique global $C^2$-solution of \eqref{dysy}. For $x^*\in \argmin g$ we introduce the function
$$h_{x^*}:[t_0,+\infty)\To\R \ h_{x^*}(t)=\frac12\|x(t)-x^*\|^2.$$
If $\alpha > 0$ and $\beta \geq 0$, then
$$\sup_{t\ge t_0}\|\dot{x}(t)\|<+\infty \ \mbox{and} \ \frac{1}{t}\|\dot{x}(t)\|^2\in L^1([t_0,+\infty),\R).$$
In addition,
$$\sup_{t\ge t_0}\frac{1}{t}|\dot{h}_{x^*}(t)|<+\infty.$$
\end{lemma}
\begin{proof}
We consider the following energy functional
\begin{equation}\label{energy}
W: [t_0,+\infty) \rightarrow \R, \ W(t)=g(x(t))+\frac12\|\dot{x}(t)\|^2+\frac{\e(t)}{2}\|x(t)\|^2.
\end{equation}
By using \eqref{dysy} we have for every $t \geq t_0$
\begin{align*}
\dot{W}(t)=& \  \<\n g(x(t),\dot{x}(t)\>+\<\ddot{x}(t),\dot{x}(t)\>+\frac{\dot \e(t)}{2}\|x(t)\|^2+\e(t)\<\dot{x}(t),x(t)\>\\
=& \ \<\n g(x(t),\dot{x}(t)\>+\frac{\dot \e(t)}{2}\|x(t)\|^2+\e(t)\<\dot{x}(t),x(t)\>\\
&\ +\left\<-\frac{\a}{t}\dot{x}(t)-\b\n^2g(x(t))\dot{x}(t)-\n g(x(t))-\e(t)x(t),\dot{x}(t)\right\>\\
= & \ -\frac{\a}{t}\|\dot{x}(t)\|^2+\frac{\dot \e(t)}{2}\|x(t)\|^2-\b\<\n^2 g(x(t))\dot{x}(t),\dot{x}(t)\>.
\end{align*}
From here, invoking the convexity of $g$, it follows
\begin{equation}\label{descenerg}
\dot{W}(t)\le -\frac{\a}{t}\|\dot{x}(t)\|^2+\frac{\dot\e(t)}{2}\|x(t)\|^2,
\end{equation}
for every $t \geq t_0$. Since  $\e$ is nonincreasing this leads further to
\begin{equation}\label{forderiv}
 \dot{W}(t)\le  -\frac{\a}{t}\|\dot{x}(t)\|^2\mbox{ for every }t\ge t_0,
 \end{equation}
therefore the energy $W$ is nonincreasing. Since $W$ is bounded from bellow, there exists
$\lim_{t\To+\infty}W(t)\in\R.$ Consequently, $t \mapsto W(t)$ is bounded on $[t_0,+\infty)$ from which, since $g$ is bounded from bellow, we obtain that
$$\sup_{t\ge t_0}\|\dot{x}(t)\|=K<+\infty.$$
By integrating \eqref{forderiv} on an interval $[t_0,t]$ for arbitrary $t > t_0$ it yields
$$\int_{t_0}^t \frac{\a}{s}\|\dot{x}(s)\|^2 ds\le W(t_0)-W(t),$$
which, by letting $t\To+\infty$, leads to
$$\frac{1}{t}\|\dot{x}(t)\|^2\in L^1([t_0,+\infty),\R).$$
Further, for every $t \geq t_0$ we have that
$$|\dot{h}_{x^*}(t)|=|\<\dot{x}(t),x(t)-x^*\>|\le\|\dot{x}(t)\|\|x(t)-x^*\|$$
and
$$\||x(t)-x^*\|\le \|x(t)-x(t_0)\|+\|x(t_0)-x^*\|\le \sup_{t\ge t_0}\|\dot{x}(t)\|(t-t_0)+\|x(t_0)-x^*\|,$$
hence, 
\begin{align*}
\frac{1}{t}|\dot{h}_{x^*}(t)|\le & \sup_{t\ge t_0}\|\dot{x}(t)\| \left(\sup_{t\ge t_0}\|\dot{x}(t)\|\left(1-\frac{t_0}{t}\right)+\frac{1}{t}\|x(t_0)-x^*\| \right)\\
\le & \sup_{t\ge t_0}\|\dot{x}(t)\| \left(\sup_{t\ge t_0}\|\dot{x}(t)\|+\frac{1}{t_0}\|x(t_0)-x^*\| \right) \in \R.
\end{align*}
\end{proof}

For each $\e > 0,$ we denote by $x_{\e}$ the unique solution of the strongly convex minimization problem
$$x_{\e} = \argmin_{x\in\mathcal{H}}\left(g(x)+\frac{\e}{2}\|x\|^2\right).$$
In virtue of the Fermat rule,  this is equivalent to
$$\n g(x_\e)+\e x_\e=0.$$
It is well known that the Tikhonov approximation curve $\e\To x_\e$ satisfies $\lim_{\e\To 0}x_\e=x^*$,
where $x^*= \argmin \{\|x\|: x \in \argmin g\}$ is the element of minimum norm of the nonempty convex closed set $\argmin g$. Since $\n g$ is monotone, for every $\e >0$ it holds $\<\n g(x_\e)-\n g(x^*), x_\e-x^*\>\ge 0$, that is $\<-\e x_\e,x_\e-x^*\>\ge 0$.
Hence,$-\|x_\e\|^2+\<x_\e,x^*\>\ge 0$, which, by using the Cauchy-Schwarz inequality, implies
$$\|x_\e\|\le\|x^*\| \ \mbox{for every} \ \e >0.$$

\subsection{Strong ergodic convergence}\label{subsec41}

We will start by proving a strong ergodic convergence result for the trajectory of \eqref{dysy}. 

\begin{theorem}\label{ergconvergence}
Let $x$ be the unique global $C^2$-solution of \eqref{dysy}. Assume that
$$\int_{t_0}^{+\infty}\frac{\e(t)}{t}dt=+\infty.$$
Let $x^*= \argmin \{\|x\|: x \in \argmin g\}$ be the element of minimum norm of the nonempty convex closed set $\argmin g$. If $\alpha > 0$, then
$$\lim_{t\To+\infty}\frac{1}{\int_{t_0}^{t}\frac{\e(s)}{s}ds}\int_{t_0}^{t}\frac{\e(s)}{s}\|x(s)-x^*\|^2 ds=0 \ \mbox{and} \ \liminf_{t\To+\infty}\|x(t)-x^*\|=0.$$
\end{theorem}
\begin{proof}
We introduce the function
$$h_{x^*}:[t_0,+\infty)\To\R,\,h_{x^*}(t)=\frac12\|x(t)-x^*\|^2.$$
For every  $t\ge t_0$ we have
\begin{equation}\label{forh}
\ddot{h}_{x^*}(t)+\frac{\a}{t}\dot{h}_{x^*}(t)=\|\dot{x}(t)\|^2+\left\<\ddot{x}(t)+\frac{\a}{t}\dot{x}(t),x(t)-x^*\right\>.
\end{equation}
Further, for every $t\ge t_0$, the function
$g_t:\mathcal{H}\To\R,\, g_t(x)=g(x)+\frac{\e(t)}{2}\|x\|^2,$ is strongly convex, with modulus $\e(t)$, hence
\begin{equation}\label{forft}
g_t(x^*)-g_t(x(t))\ge \<\n g_t(x(t)), x^*-x(t)\>+\frac{\e(t)}{2}\|x(t)-x^*\|^2.
\end{equation}
But $\n g_t(x(t))=\n g(x(t))+\e(t) x(t)$ and by using \eqref{dysy} we get
$$\n g_t(x(t))=-\ddot{x}(t)-\frac{\a}{t}\dot{x}(t)-\b\n^2 g(x(t))\dot{x}(t) \ \mbox{for every} \ t \geq t_0.$$
Consequently, \eqref{forft} becomes
\begin{equation}\label{forft1}
g_t(x^*)-g_t(x(t))\ge \left\<\ddot{x}(t)+\frac{\a}{t}\dot{x}(t)+\b\n^2 g(x(t))\dot{x}(t), x(t)-x^*\right\>+\frac{\e(t)}{2}\|x(t)-x^*\|^2  \ \mbox{for every} \ t \geq t_0.
\end{equation}
By using \eqref{forh}, the latter relation leads to
\begin{equation}\label{forft2}
g_t(x^*)-g_t(x(t))\ge \ddot{h}_{x^*}(t)+\frac{\a}{t}\dot{h}_{x^*}(t)+\e(t)h_{x^*}(t)+\<\b\n^2 g(x(t))\dot{x}(t), x(t)-x^*\>-\|\dot{x}(t)\|^2
\end{equation}
for every $t \geq t_0$.

For every $t \geq t_0$, let $x_{\e(t)}$ the unique solution of the strongly convex minimization problem
$$\min_{x\in\mathcal{H}}\left(g(x)+\frac{\e(t)}{2}\|x\|^2\right).$$
Then
$$g_t(x^*)-g_t(x(t))\le g_t(x^*)-g_t(x_{\e(t)})=g(x^*)+\frac{\e(t)}{2}\|x^*\|^2-g(x_{\e(t)})-\frac{\e(t)}{2}\|x_{\e(t)}\|^2 \le\frac{\e(t)}{2}(\|x^*\|^2-\|x_{\e(t)}\|^2)$$
for every $t \geq t_0$ and taking into account  \eqref{forft2} we get
\begin{equation}\label{forft4}
\frac{\e(t)}{2}(\|x^*\|^2-\|x_{\e(t)}\|^2)\ge \ddot{h}_{x^*}(t)+\frac{\a}{t}\dot{h}_{x^*}(t)+\e(t)h_{x^*}(t)+\<\b\n^2 g(x(t))\dot{x}(t), x(t)-x^*\>-\|\dot{x}(t)\|^2
\end{equation}
for every $t \geq t_0$. We have
$$\ddot{h}_{x^*}(t)+\frac{\a}{t}\dot{h}_{x^*}(t)=\frac{1}{t^\a}\frac{d}{dt}\left(t^\a\dot{h}_{x^*}(t)\right)$$
and
$$\<\n^2 g(x(t))\dot{x}(t), x(t)-x^*\>=\frac{d}{dt}\big(\<\n g(x(t)), x(t)-x^*\>-g(x(t))\big)$$
hence \eqref{forft4} is equivalent to
\begin{equation}\label{finft}
\frac{\e(t)}{t}\left(h_{x^*}(t)-\frac{1}{2}(\|x^*\|^2-\|x_{\e(t)}\|^2)\right)\le \frac{1}{t}\|\dot{x}(t)\|^2-\frac{1}{t^{\a+1}}\frac{d}{dt}(t^{\a}\dot{h}_{x^*}(t))-\frac{\b}{t}\frac{d}{dt}(\<\n g(x(t)), x(t)-x^*\>-g(x(t))),
\end{equation}
for every $t \geq t_0$.

After integrating \eqref{finft} on $[t_0,t]$, for arbitrary $t > t_0$, it yields
\begin{align}\label{er1}\int_{t_0}^t \frac{\e(s)}{s}\left(h_{x^*}(s)-\frac12(\|x^*\|^2-\|x_{\e(s)}\|^2)\right)ds
\le& \ \int_{t_0}^t\left( \frac{1}{s}\|\dot{x}(s)\|^2-\frac{1}{s^{\a+1}}\frac{d}{ds}\left(s^{\a}\dot{h}_{x^*}(s)\right)\right)ds \nonumber \\
&\ +\int_{t_0}^t \frac{\b}{s}\frac{d}{ds}\left(\<\n g(x(s)), x^*-x(s)\>+g(x(s))\right)ds.
\end{align}
We show that the right-hand side of the above inequality is bounded from above. Indeed, according to Lemma \ref{ergl}, one has $$\frac{1}{t}\|\dot{x}(t)\|^2\in L^1([t_0,+\infty),\R),$$
hence there exists $C_1\ge 0$ such that $\int_{t_0}^t \frac{1}{s}\|\dot{x}(s)\|^2\le C_1$ for every $t\ge t_0$. Further, for every $t \geq t_0$,
\begin{align*}\int_{t_0}^t \frac{1}{s^{\a+1}}\frac{d}{ds}(s^{\a}\dot{h}_{x^*}(s))ds
&=\frac{\dot{h}_{x^*}(t)}{t}-\frac{\dot{h}_{x^*}(t_0)}{t_0}+(\a+1)\int_{t_0}^t\frac{\dot{h}_{x^*}(s)}{s^2}ds\\
&=\frac{\dot{h}_{x^*}(t)}{t}-\frac{\dot{h}_{x^*}(t_0)}{t_0}+(\a+1)\left(\frac{{h}_{x^*}(t)}{t^2}-\frac{{h}_{x^*}(t_0)}{t_0^2}\right)+2(\a+1)\int_{t_0}^t\frac{{h}_{x^*}(s)}{s^3}ds\\
&\ge\frac{\dot{h}_{x^*}(t)}{t}-C_2,
\end{align*}
where $C_2=\frac{\dot{h}_{x^*}(t_0)}{t_0}+(\a+1)\frac{{h}_{x^*}(t_0)}{t_0^2}.$ Consequently, 
\begin{align}\label{er2}\int_{t_0}^t \frac{\e(s)}{s}\left(h_{x^*}(s)-\frac12(\|x^*\|^2-\|x_{\e(s)}\|^2)\right)ds \le& \ C_1+C_2-\frac{\dot{h}_{x^*}(t)}{t}
\nonumber\\
&+\int_{t_0}^t \frac{\b}{s}\frac{d}{ds}(\<\n g(x(s)), x^*-x(s)\>+g(x(s)))ds,
\end{align}
for every $t \geq t_0$. According to Lemma \ref{ergl}, there exists $C_3$ such that
$\frac{1}{t}|\dot{h}_{x^*}(t)|\le C_3\mbox{ for all }t\ge t_0$, which combined with \eqref{er2} guarantees the existence of $C_4\ge 0$ such that
\begin{equation}\label{er3}\int_{t_0}^t \frac{\e(s)}{s}\left(h_{x^*}(s)-\frac12(\|x^*\|^2-\|x_{\e(s)}\|^2)\right)ds\le C_4+\int_{t_0}^t \frac{\b}{s}\frac{d}{ds}\left(\<\n g(x(s)), x^*-x(s)\>+g(x(s))\right)ds
\end{equation}
for every $t\ge t_0$.

On the other hand, for every $t \geq t_0$,
\begin{align*}
\int_{t_0}^t \frac{\b}{s}\frac{d}{ds}(\<\n g(x(s)), x^*-x(s)\>+g(x(s)))ds= & \int_{t_0}^t \frac{\b}{s^2}\big(\<\n g(x(s)), x^*-x(s)\>+g(x(s))\big)ds \\
& + \frac{\b}{t}\big(\<\n g(x(t)), x^*-x(t)\>+g(x(t))\big)\\
& -\frac{\b}{t_0}\big(\<\n g(x(t_0)), x^*-x(t_0)\>+g(x(t_0))\big).
\end{align*}
From the gradient inequality of the convex function $g$ we have
$$\<\n g(x(t)), x^*-x(t)\>+g(x(t))\le g(x^*),$$
hence
\begin{align}\label{forgrad}
\int_{t_0}^t \frac{\b}{s}\frac{d}{ds}(\<\n g(x(s)), x^*-x(s)\>+g(x(s)))ds& \ \le \frac{\b}{t}g(x^*) +\int_{t_0}^t \frac{\b}{s^2}g(x^*)ds \nonumber \\
&   \ -\frac{\b}{t_0}(\<\n g(x(t_0)), x^*-x(t_0)\>+g(x(t_0))),
\end{align}
for all $t \geq t_0$. Obviously the right-hand side of \eqref{forgrad} is bounded from above, hence there exists $C_5>0$ such that
\begin{align}\label{forgrad1}
\int_{t_0}^t \frac{\b}{s}\frac{d}{ds}(\<\n g(x(s)), x^*-x(s)\>+g(x(s)))ds&\le C_5\mbox{ for every }t\ge t_0.
\end{align}
Combining \eqref{er3} and \eqref{forgrad1} we obtain that there exists $C>0$ such that
\begin{align}\label{er4}\int_{t_0}^t \frac{\e(s)}{s}\left(h_{x^*}(s)-\frac12(\|x^*\|^2-\|x_{\e(s)}\|^2)\right)ds&\le  C \mbox{ for every }t\ge t_0.
\end{align}
Since $\lim_{t\To+\infty}\e(t)=0$ we have $\lim_{t\To+\infty}x_{\e(t)}=x^*$, hence
$\lim_{t\To+\infty}(\|x^*\|^2-\|x_{\e(t)}\|^2)=0.$
Consequently, by using the l'Hospital rule and the fact that
$\int_{t_0}^{+\infty} \frac{\e(t)}{t}dt=+\infty$, we get
$$\lim_{t\To+\infty}\!\frac{1}{\int_{t_0}^{t}\frac{\e(s)}{s}ds}\int_{t_0}^{t}\frac{\e(s)}{s}(\|x^*\|^2-\|x_{\e(s)}\|^2)ds=\!\!\lim_{t\To+\infty}\!\frac{\frac{\e(t)}{t}(\|x^*\|^2-\|x_{\e(t)}\|^2)}{\frac{\e(t)}{t}}=\!\!\lim_{t\To+\infty}(\|x^*\|^2-\|x_{\e(t)}\|^2)=0.$$
Dividing \eqref{er4} by $\int_{t_0}^t \frac{\e(s)}{s}ds$ and taking into account that
$\int_{t_0}^{+\infty} \frac{\e(t)}{t}dt=+\infty$,  we obtain that
$$\lim_{t\To+\infty}\frac{1}{\int_{t_0}^{t}\frac{\e(s)}{s}ds}\int_{t_0}^{t}\frac{\e(s)}{s}\|x(s)-x^*\|^2 ds=0.$$
The last equality immediately implies that 
$$\liminf_{t\To+\infty}\|x(t)-x^*\|=0.$$
\end{proof}

\begin{remark}
The strong ergodic convergence obtained in \cite{ACR} for the dynamical system \eqref{dysy-sbc-tikh} is extended to the dynamical system with Hessian driven damping and Tikhonov regularization term \eqref{dysy} under the same hypotheses concerning the Tikhonov parametrization $t \mapsto \e(t)$.  
\end{remark}

\subsection{Strong convergence}\label{sec42}

In order to prove strong convergence for the trajectory generated by the dynamical system \eqref{dysy} to an element of minimum norm of $\argmin g$ we have to strengthen the conditions on the Tikhonov parametrization. This is done in the following result.

\begin{theorem}\label{strongconvergence}
Let be $\a \geq 3$ and $x$ the unique global $C^2$-solution of \eqref{dysy}. Assume that
$$\int_{t_0}^{+\infty} \frac{\e(t)}{t}dt<+\infty \ \mbox{and} \ \lim_{t \To +\infty} \frac{\b}{\e(t)t^{{\frac{\alpha}{3}+1}}}\int_{t_0}^t \e^2(s)s^{\frac{\alpha}{3}+1}ds=0,$$
and that there exist $a>1$ and $t_1\ge t_0$ such that
$$\dot{\e}(t)\le-\frac{a\b}{2}\e^2(t) \ \mbox{ for every }t\ge t_1.$$
In addition, assume that
\begin{itemize}
\item in case $\a=3$: $\lim_{t\To+\infty}t^2\e(t) =+\infty$;
\item in case $\a >3$: there exists $c>0$ such that $t^2\e(t)\ge\frac23 \a\left(\frac13\a-1+\b c^2\right)$ for $t$ large enough.
\end{itemize}
If $x^*= \argmin \{\|x\|: x \in \argmin g\}$ is the element of minimum norm of the nonempty convex closed set $\argmin g$, then
$$\liminf_{t\To+\infty}\|x(t)-x^*\|=0.$$
In addition,
$$\lim_{t\To+\infty}\|x(t)-x^*\|=0,$$
if there exists $T\geq t_0$ such that the trajectory $\{x(t) :t \ge T\}$ stays either in the ball $B(0, \|x^*\|)$, or in its complement.
\end{theorem}
\begin{proof} 
{\it Case I. Assume that there exists $T \geq t_0$ such that the trajectory $\{x(t) :t \ge T\}$ stays in the complement of the ball $B(0, \|x^*\|)$}

In other words, $\|x(t)\| \geq \|x^*\|$ for every $t \geq T$. For $p\ge 0$, we consider the energy functional 
\begin{align}\label{energynew}
{\cal E}_b^p(t)=& \ t^{p+1}(t+\a-\b-\b p-b-1)(g(x(t))-\min g)+t^{p+2}\frac{\e(t)}{2}(\|x(t)\|^2-\|x^*\|^2) \nonumber\\ 
& \ +\frac{t^p}{2}\|b(x(t)-x^*)+t(\dot{x}(t)+\b\n g(x(t)))\|^2 \mbox{ for every } t\ge t_0.
\end{align}
We define $t_2:=\max\left(t_1, 2(\b+\b p+b+1-\a)\right)$. We have that
\begin{align}\label{forE}
{\cal E}_b^p(t)& \ge t^{p+1}(t+\a-\b-\b p-b-1)(g(x(t))-\min g)+t^{p+2}\frac{\e(t)}{2}(\|x(t)\|^2-\|x^*\|^2)
\nonumber\\ 
& \ge t^{p+2}\frac12(g(x(t))-\min g)+t^{p+2}\frac{\e(t)}{2}(\|x(t)\|^2-\|x^*\|^2)  \mbox{ for every }t\ge t_2.
\end{align}

For every $t\ge t_0$ consider the strongly convex function
$$g_t : \mathcal{H} \To \R, \ g_t(x)=\frac12 g(x)+\frac{\e(t)}{2}\|x\|^2,$$
and denote
$$x_{\e(t)}:=\argmin_{x\in\mathcal{H}}g_t(x).$$
Since $x^*$ is the element of minimum norm in $\argmin \frac12 g = \argmin g$, it holds$\|x_{\e(t)}\| \le\|x^*\|$. Using the gradient inequality we have
$$g_t(x)-g_t(x_{\e(t)})\ge \frac{\e(t)}{2}\|x-x_{\e(t)}\|^2 \ \mbox{for every} \ x \in \mathcal{H}.$$
On the other hand,
$$g_t(x_{\e(t)})-g_t(x^*)=\frac12(g(x_{\e(t)})-\min g)+\frac{\e(t)}{2}(\|x_{\e(t)}\|^2-\|x^*\|^2)\ge \frac{\e(t)}{2}(\|x_{\e(t)}\|^2-\|x^*\|^2).$$
By adding the last two inequalities we obtain
\begin{equation}\label{forf}
g_t(x)-g_t(x^*)\ge \frac{\e(t)}{2}(\|x-x_{\e(t)}\|^2+\|x_{\e(t)}\|^2-\|x^*\|^2) \ \mbox{for every} \ x \in \mathcal{H}.
\end{equation}
From \eqref{forE} and \eqref{forf} we have that for every $t\ge t_2$ it holds
\begin{equation}\label{usefulE}
{\cal E}_b^p(t)\ge t^{p+2}(g_t(x(t))-g_t(x^*))\ge \frac{\e(t)}{2}t^{p+2}\big(\|x(t)-x_{\e(t)}\|^2+\|x_{\e(t)}\|^2-\|x^*\|^2\big).
\end{equation}

The next step is to obtain an upper bound for $t \mapsto E_b^p(t)$, and to this end we will evaluate its time derivative. For every $t\ge t_0$ we have
\begin{align}\label{E1}
\frac{d}{dt}{\cal E}_b^p(t)= & \ t^{p}((p+2)t+(p+1)(\a-\b-\b p-b-1))(g(x(t))-\min g)\nonumber\\
\nonumber &+t^{p+1}(t+\a-\b-\b p-b-1)\<\n g(x(t)),\dot{x}(t)\>\\
\nonumber &+ \left((p+2)t^{p+1}\frac{\e(t)}{2}+t^{p+2}\frac{\dot{\e}(t)}{2}\right)(\|x(t)\|^2-\|x^*\|^2)+t^{p+2}\e(t))\<\dot{x}(t),x(t)\>\\
\nonumber & + \frac{p t^{p-1}}{2}\|b(x(t)-x^*)+t(\dot{x}(t)+\b\n g(x(t)))\|^2\\
& + t^p\<(b+1)\dot{x}(t)+\b\n g(x(t))+t(\ddot{x}(t)+\b\n^2 g(x(t))\dot{x}(t)),b(x(t)-x^*)+t(\dot{x}(t)+\b\n g(x(t)))\>.
\end{align}
By using \eqref{dysy} we have
$$\ddot{x}(t)+\b\n^2 g(x(t))\dot{x}(t)=-\frac{\a}{t}\dot{x}(t)-\n g(x(t))-\e(t)x(t),$$
hence
\begin{align}
\label{useE1} &\<(b+1)\dot{x}(t)+\b\n g(x(t))+t(\ddot{x}(t)+\b\n^2 g(x(t))\dot{x}(t)),b(x(t)-x^*)+t(\dot{x}(t)+\b\n g(x(t)))\> \nonumber \\
\nonumber =& \ \<(b+1-\a)\dot{x}(t)+\b\n g(x(t))-t(\n g(x(t))+\e(t)x(t)),b(x(t)-x^*)+t(\dot{x}(t)+\b\n g(x(t)))\>\\
\nonumber =& \ b(b+1-\a)\<\dot{x}(t),x(t)-x^*\>+(b+1-\a)t(\|\dot{x}(t)\|^2+\<\n g(x(t)),\dot{x}(t)\>)\\
\nonumber & +\b b\<\n g(x(t),x(t)-x^*\>+\b t\<\n g(x(t)),\dot{x}(t)\>+\b^2 t\|\n g(x(t))\|^2\\
&  -bt\<\n g(x(t))+\e(t)x(t),x(t)-x^*\>-t^2\<\n g(x(t))+\e(t)x(t),\dot{x}(t)\>\!-\b t^2\<\n g(x(t))+\e(t)x(t),\n g(x(t))\>
\end{align}
for every $t \geq t_0$. Further, for every $t \geq t_0$,
\begin{align}
\label{useE2} \|b(x(t)-x^*)+t(\dot{x}(t)+\b\n g(x(t)))\|^2= & \ b^2\|x(t)-x^*\|^2+2b t\<\dot{x}(t),x(t)-x^*\>+2b\b t\<\n g(x(t)),x(t)-x^*\>\nonumber \\
& +t^2\|\dot{x}(t)\|^2+2\b t^2\<\n g(x(t)),\dot{x}(t)\>+\b^2 t^2\|\n g(x(t))\|^2,
\end{align}
which means that \eqref{E1} becomes
\begin{align}\label{E2}
\frac{d}{dt}{\cal E}_b^p(t)  = & \ t^{p}((p+2)t+(p+1)(\a-\b-\b p-b-1))(g(x(t))-\min g) \nonumber \\
\nonumber &+ \left((p+2)t^{p+1}\frac{\e(t)}{2}+t^{p+2}\frac{\dot{\e}(t)}{2}\right)(\|x(t)\|^2-\|x^*\|^2)+ \frac{b^2 p t^{p-1}}{2}\|x(t)-x^*\|^2\\
\nonumber & +\frac{(p+2)\b^2 t^{p+1}}{2}\|\n g(x(t))\|^2 +\left(b+1-\a+\frac{p}{2}\right) t^{p+1}\|\dot{x}(t)\|^2\\
\nonumber & +b(b+1-\a+p) t^{p}\<\dot{x}(t),x(t)-x^*\>+b\b( p+1) t^{p}\<\n g(x(t)),x(t)-x^*\>\\
&-b t^{p+1}\<\n g(x(t))+\e(t)x(t),x(t)-x^*\>-\b t^{p+2}\<\n g(x(t))+\e(t)x(t),\n g(x(t))\>.
\end{align}

The gradient inequality for the strongly convex function $x\to g(x)+\frac{\e(t)}{2}\|x\|^2$ gives
$$\<\n g(x(t))+\e(t)x(t),x^*-x(t)\>+\frac{\e(t)}{2}\|x(t)-x^*\|^2\le \left(g(x^*)+\frac{\e(t)}{2}\|x^*\|^2\right)-\left(g(x(t))+\frac{\e(t)}{2}\|x(t)\|^2\right),$$
hence
\begin{align*}
-b t^{p+1}\<\n g(x(t))+\e(t)x(t),x(t)-x^*\>\le & -b t^{p+1}(g(x(t))-g^*)\\
&- b t^{p+1}\frac{\e(t)}{2}(\|x(t)\|^2-\|x^*\|^2)-b t^{p+1}\frac{\e(t)}{2}\|x(t)-x^*\|^2
\end{align*}
for every $t \geq t_0$. Plugging this inequality into \eqref{E2} gives
\begin{align}\label{E3}
\frac{d}{dt}{\cal E}_b^p(t)  \le & \ t^{p}((p+2-b)t+(p+1)(\a-\b-\b p-b-1))(g(x(t))-\min g) \nonumber\\
\nonumber &+ \left((p+2-b)t^{p+1}\frac{\e(t)}{2}+t^{p+2}\frac{\dot{\e}(t)}{2}\right)(\|x(t)\|^2-\|x^*\|^2)+ \left(\frac{b^2 p t^{p-1}}{2}-b t^{p+1}\frac{\e(t)}{2}\right)\|x(t)-x^*\|^2\\
\nonumber & +\left(\frac{(p+2)\b^2 t^{p+1}}{2}-\b t^{p+2}\right)\|\n g(x(t))\|^2 +\left(b+1-\a+\frac{p}{2}\right) t^{p+1}\|\dot{x}(t)\|^2\\
\nonumber & +b(b+1-\a+p) t^{p}\<\dot{x}(t),x(t)-x^*\>+b\b( p+1) t^{p}\<\n g(x(t)),x(t)-x^*\>\\
&-\b t^{p+2}\e(t)\<\n g(x(t)),x(t)\>
\end{align}
for every $t \geq t_0$.
Further we have for every $t \geq t_0$
\begin{align}\label{useE3}
b\b( p+1) t^{p}\<\n g(x(t)),x(t)-x^*\>\leq & \ \frac{b\b(p+1)}{4c^2}t^{p+1}\|\n g(x(t))\|^2+b\b(p+1)c^2t^{p-1}\|x(t)-x^*\|^2
\end{align}
 and
\begin{align}\label{useE4}
-\b t^{p+2}\e(t)\<\n g(x(t)),x(t)\>\leq & \ \frac{\b}{a}t^{p+2}\|\n g(x(t))\|^2+\frac{a\b}{4}\e^2(t)t^{p+2}\|x(t)\|^2,
\end{align}
where $a>1$ and $c>0$ are the constants which are assumed to exist in the hypotheses of the theorem, whereby in case $\alpha=3$ we will take $c=1$.

Combining \eqref{E3}, \eqref{useE3} and \eqref{useE4} and neglecting the nonpositive terms we derive 
\begin{align}\label{E4}
\frac{d}{dt}{\cal E}_b^p(t)  \le & \ t^{p}((p+2-b)t+(p+1)(\a-\b-\b p-b-1))(g(x(t))-\min g) \nonumber\\
\nonumber &+ \left((p+2-b)t^{p+1}\frac{\e(t)}{2}+t^{p+2}\frac{\dot{\e}(t)}{2}+\frac{a\b}{4}\e^2(t)t^{p+2}\right)\|x(t)\|^2\\
\nonumber &-\left((p+2-b)t^{p+1}\frac{\e(t)}{2}+t^{p+2}\frac{\dot{\e}(t)}{2}\right)\|x^*\|^2\\
\nonumber &+ \left(\frac{b^2 p t^{p-1}}{2}+b\b(p+1)c^2t^{p-1}-b t^{p+1}\frac{\e(t)}{2}\right)\|x(t)-x^*\|^2\\
\nonumber & +\left(\frac{(p+2)\b^2 t^{p+1}}{2}+\frac{b\b(p+1)}{4c^2}t^{p+1}-\b\left(1-\frac{1}{a}\right) t^{p+2}\right)\|\n g(x(t))\|^2 \\
& +\left(b+1-\a+\frac{p}{2}\right) t^{p+1}\|\dot{x}(t)\|^2+b(b+1-\a+p) t^{p}\<\dot{x}(t),x(t)-x^*\>
\end{align}
for every $t \geq t_0$.

For the remaining of the proof we choose the parameters appearing in the definition of the energy functional as
$$b:=\frac23\a \ \mbox{and} \ p:=\frac13(\a-3).$$
Since $\a\ge 3$, we have
$$p+2-b=1-\frac{\a}{3}\le0,\,b+1+p-\a=0\mbox{ and }b+1+\frac{p}{2}-\a=-\frac{p}{2}\le 0.$$
Notice that, if $\a=3$, then $(p+2-b)t+(p+1)(\a-\b-\b p-b-1)=-\b\leq0$ and, if $\a>3$, then $p+2-b<0$. This means that there exists $t_3\ge t_2$ such that
$(p+2-b)t+(p+1)(\a-\b-\b p-b-1)<0$ for every $t\ge t_3.$ This implies that the term
$$t^{p}((p+2-b)t+(p+1)(\a-\b-\b p-b-1))(g(x(t))-\min g)$$ in \eqref{E4} is nonpositive for every $t\ge t_2$ and therefore we will omit it. Further, using that  $\lim_{t\To+\infty}t^2\e(t) =+\infty$, if $\a=3$, and that $t^2\e(t)\ge\frac23 \a(\frac13\a-1+\b c^2)$ for $t$ large enough,  if $\a>3,$ we immediately see that there exists $t_4\ge t_3$ such that
$$\frac{b^2 p t^{p-1}}{2}+b\b(p+1)c^2t^{p-1}-b t^{p+1}\frac{\e(t)}{2}\le0 \mbox{ for every }t\ge t_3.$$
Finally, since $a>1$, it is obvious that there exists $t_5\ge t_4$ such that
$$\frac{(p+2)\b^2 t^{p+1}}{2}+\frac{b\b(p+1)}{4c^2}t^{p+1}-\b\left(1-\frac{1}{a}\right) t^{p+2}\le0 \mbox{ for every }t\ge t_5.$$
Thus, \eqref{E4} yields
\begin{align}\label{E5}
\frac{d}{dt}{\cal E}_b^p(t) \le  & \ \left((p+2-b)t^{p+1}\frac{\e(t)}{2}+t^{p+2}\frac{\dot{\e}(t)}{2}+\frac{a\b}{4}\e^2(t)t^{p+2}\right)\|x(t)\|^2 \nonumber \\
\nonumber &-\left((p+2-b)t^{p+1}\frac{\e(t)}{2}+t^{p+2}\frac{\dot{\e}(t)}{2}\right)\|x^*\|^2\\ = & \left((p+2-b)t^{p+1}\frac{\e(t)}{2}+t^{p+2}\frac{\dot{\e}(t)}{2}+\frac{a\b}{4}\e^2(t)t^{p+2}\right)(\|x(t)\|^2-\|x^*\|^2)+
\frac{a\b}{4}\e^2(t)t^{p+2}\|x^*\|^2,
\end{align}
for every $t\ge t_5$. By the hypotheses, we have that
$$(p+2-b)t^{p+1}\frac{\e(t)}{2}+t^{p+2}\frac{\dot{\e}(t)}{2}+\frac{a\b}{4}\e^2(t)t^{p+2}\le 0,$$
for every $t\ge t_5$ and, taking into account the setting considered in this first case, it follows there exists $t_6\ge t_5$ such that
$$\|x(t)\|^2-\|x^*\|^2\ge 0$$
for every $t\ge t_6$. Hence, \eqref{E5} leads to
\begin{align}\label{E51}
\frac{d}{dt}{\cal E}_b^p(t) & \le \frac{a\b}{4}\e^2(t)t^{p+2}\|x^*\|^2 \ \mbox{for every} \ t \geq t_6.
\end{align}
 By integrating \eqref{E51} on the interval $[t_6, t]$, for arbitrary $t \geq t_6$, we get
\begin{align}\label{E6}
{\cal E}_b^p(t) &\le {\cal E}_b^p(t_6)+\frac{a\b}{4}\|x^*\|^2 \int_{t_6}^t \e^2(s)s^{p+2}dt.
\end{align}
Recall that from \eqref{usefulE} we have
$${\cal E}_b^p(t)\ge \frac{\e(t)}{2}t^{p+2}(\|x(t)-x_{\e(t)}\|^2+\|x_{\e(t)}\|^2-\|x^*\|^2),$$
which, combined with \eqref{E6}, gives for every $t\ge t_6$ that
\begin{align}\label{xe1}
\|x(t)-x_{\e(t)}\|^2 & \le \|x^*\|^2-\|x_{\e(t)}\|^2+\frac{2E_b^p(t_6)}{\e(t)t^{\frac13\a+1}} +\frac{a\b}{2\e(t)t^{\frac13\a+1}}\|x^*\|^2 \int_{t_6}^t \e^2(s)s^{\frac13\a+1}dt.
\end{align}

Using that $\lim_{t\To+\infty}\e(t)t^{\frac13\a+1}=+\infty$, $\lim_{t\To+\infty}x_{\e(t)}=x^*$ and taking into account the hypotheses of the theorem, we get that the right-hand side of \eqref{xe1} converges to $0$ as $t \To+\infty$. This yields
$$\lim_{t\To+\infty} x(t)=x^*.$$

{\it Case II. Assume that there exists $T \geq t_0$ such that the trajectory $\{x(t) :t \ge T\}$ stays in the ball $B(0, \|x^*\|)$}

In other words, $\|x(t)\| < \|x^*\|$ for every $t \geq T$. Since 
$$\int_{t_0}^{+\infty}\frac{\e(t)}{t}dt<+\infty,$$
according to Theorem \ref{conv}, we have
$$\lim_{t\To+\infty}g(x(t))=\min g.$$
Consider $\ol x \in \mathcal{H}$ a weak sequential cluster point of the trajectory $x$, which exists since the trajectory is bounded. This means that there exists a sequence 
$(t_n)_{n\in\N} \subseteq [T,+\infty)$ such that $t_n\To+\infty$ and $x(t_n)$ converges weakly to $\ol x$ as $n \To +\infty$.

Since $g$ is weakly lower semicontinuous, it holds
$$g(\ol x)\le\liminf_{n\to+\infty} g(x(t_n))=\min g, \ \mbox{thus} \ \ol x\in\argmin g.$$
Since the norm is weakly lower semicontinuous, it holds
$$\|\ol x\|\le\liminf_{n\to+\infty}\|x(t_n)\|\le\|x^*\|,$$
which, by taking into account that $x^*$ is the unique element of minimum norm in $\argmin g$, implies $\ol x=x^*$. This shows that the whole trajectory $x$ converges weakly to $x^*$.

Thus,
$$\|x^*\|\leq \liminf_{t\to+\infty}\|x(t)\|\le\limsup_{t\to+\infty}\|x(t)\|\le\|x^*\|, \
\mbox{hence} \ \lim_{t\To+\infty}\|x(t)\|=\|x^*\|.$$
But by taking into account that $x(t)\rightharpoonup x^*$ as $t\To+\infty$, we obtain that  the convergence is strong, that is
$$\lim_{t\To+\infty}x(t)=x^*.$$

{\it Case III. Assume that for every $T \geq t_0$ there exists $t\ge T$ such that $\|x^*\| > \|x(t)\|$ and there exists $s\ge T$ such that $\|x^*\| \le \|x(s)\|$}

By the continuity of $x$ it follows that there exists a sequence 
$(t_n)_{n\in\N} \subseteq [t_0,+\infty)$ such that $t_n\To+\infty$ as $n \To +\infty$ and
$$\|x(t_n)\|=\|x^*\| \mbox{ for every }n\in\N.$$

We  will show that $x(t_n) \To x^*$ as $n\To+\infty.$ To this end we consider $\ol x \in \mathcal{H}$ a weak sequential cluster point of the sequence $(x(t_n))_{n\in\N}.$ By repeating the arguments used in the previous case (notice that the sequence is bounded) it follows that $(x(t_n))_{n\in\N}$ converges weakly to $x^*$ as $n \To +\infty$. Since
$\|x(t_n)\|\To\|x^*\|$ as $n\To+\infty$, it yields $\|x(t_n)- x^*\|\To 0$ as $n\To+\infty$. This shows that
$$\liminf_{t\To+\infty}\|x(t)-x^*\|=0.$$
\end{proof} 

\begin{remark}\label{rem45}
Theorem \ref{strongconvergence} can be seen as an extension of a result given in \cite{ACR} for the dynamical system \eqref{dysy-sbc-tikh} to the dynamical system with Hessian driven damping and Tikhonov regularization term \eqref{dysy}. One can notice that for the choice $\beta =0$, which means that the Hessian driven damping is removed, the lower bound we impose for $t \mapsto t^2\e(t)$ in case $\alpha >3$ is less tight than the one considered in \cite[Theorem 4.1]{ACR} for the system \eqref{dysy-sbc-tikh}. As we will see later, this lower bound influences the asymptotic behaviour of the trajectory.

In case $\beta >0$, in order to guarantee that  
$$\lim_{t \To +\infty} \frac{\b}{\e(t)t^{{\frac{\alpha}{3}+1}}}\int_{t_0}^t \e^2(s)s^{\frac{\alpha}{3}+1}ds=0,$$
one just have to additionally assume that
$$\int_{t_0}^{+\infty} \e(t)dt<+\infty$$
and that the function 
$$t\To t^{\frac{1}{3}\a+1}\e(t) \ \mbox{is nondecreasing for} \ t \ \mbox{large enough}.$$
This follows from Lemma \ref{nullimit}, by also taking into account that $\lim_{t\To+\infty}\e(t)t^{\frac{\alpha}{3}+1}=+\infty$.

Combining the main results in the last two sections, one can see that if
$$\int_{t_0}^{+\infty} t \e(t)dt<+\infty,$$
the function 
$$t\To t^{\frac{1}{3}\a+1}\e(t) \ \mbox{is nondecreasing for} \ t \ \mbox{large enough},$$
there exist $a>1$ and $t_1\ge t_0$ such that
$$\dot{\e}(t)\le-\frac{a\b}{2}\e^2(t) \ \mbox{ for every }t\ge t_1,$$
and
\begin{itemize}
\item in case $\a=3$: $\lim_{t\To+\infty}t^2\e(t) =+\infty$;
\item in case $\a >3$: there exists $c>0$ such that $t^2\e(t)\ge\frac23 \a\left(\frac13\a-1+\b c^2\right)$ for $t$ large enough,
\end{itemize}
then one obtains both fast convergence of the function values and strong convergence of the trajectory to the minimal norm solution. This is for instance the case when $\e(t) = t^{-\gamma}$ for all $\gamma \in (1,2)$.
\end{remark}
In the following, we would like to comment on the role on the condition in Theorem \ref{strongconvergence} which asks, in case $\alpha >3$, for the existence of a positive constant $c$ such that $t^2\e(t)\ge\frac23 \a(\frac13\a-1+\b c^2)$ for $t$ large enough. 
To this end it is very helpful to visualize the trajectories generated by the dynamical system \eqref{dysy} in relation with the minimization of the function given in \eqref{functiong} for a fixed large value of $\alpha$ and Tikhonov  parametrizations of the form  $t \mapsto \e(t)=t^{-\gamma}$, 
for different values of $\gamma \in (1,2)$. The trajectories in the plot in Figure \ref{fig:ex2} have been generated for $\alpha=200$ and $\beta =1$ and are all approaching the minimum norm solution $x^*=0$.  The norm of the difference between the trajectory and the minimum norm solution is guaranteed to be bounded from above by a function which converges to zero, 
after the time point $t$ is reached at which the inequality $t^2\e(t)\ge\frac23 \a(\frac13\a-1+\b c^2)$ ``starts'' being fulfilled. For large $\alpha$ and the Tikhonov parametrizations considered in our experiment, the closer $\gamma$ is to $1$ is, the faster is this inequality fulfilled. This is reflected by the behaviour of the trajectories plotted in  Figure \ref{fig:ex2}.

\begin{figure}[h]
	\centering
	\captionsetup[subfigure]{position=top}
	%\subfloat[...]
	{\includegraphics*[viewport=130 250 485 574,width=0.7\textwidth]{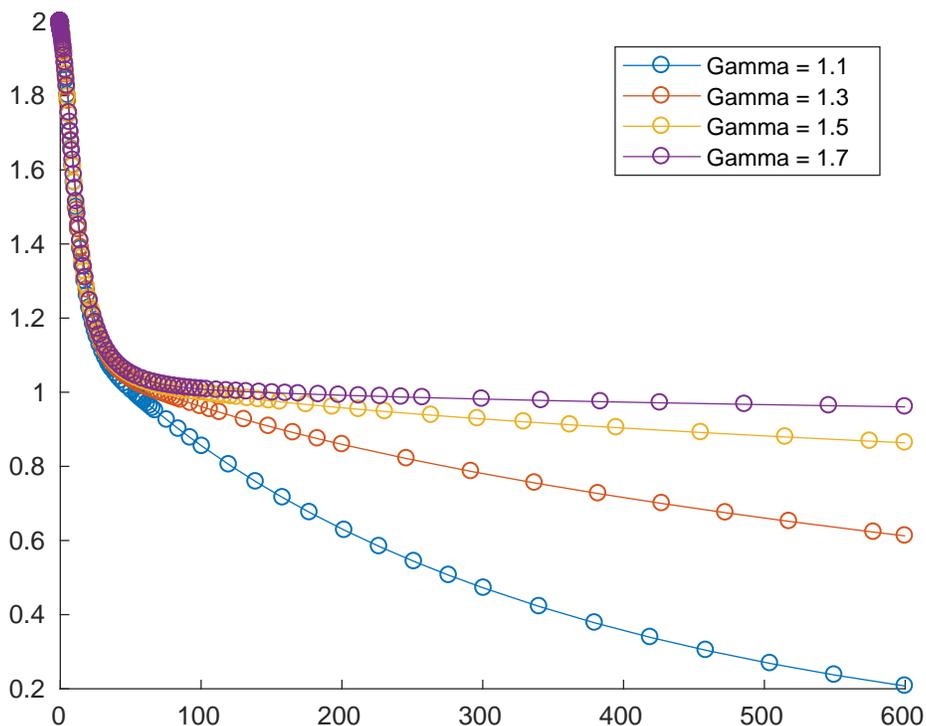}} 
	\caption{\small The behaviour of the trajectories generated by the dynamical system \eqref{dysy} in relation with the minimization of the function given in \eqref{functiong} for $\a=200$, $\beta=1$, $\e(t)=t^{-\gamma}$ and different values for $\gamma \in (1,2)$.}
\label{fig:ex2}	
\end{figure}

Finally, we would like to formulate some possible questions of future research related to the dynamical sytem \eqref{dysy}:
\begin{itemize}
\item[$\bullet$] in \cite[Theorem 3.4]{att-c-p-r-math-pr2018} it has been proved for the dynamical system \eqref{dysy-sbc} that, when $g$ is strongly convex, the rates of convergence of the function values and the tracjectory are both of $O(t^{-\frac{2}{3}\alpha})$, thus they can be made arbitrarily fast by taking $\alpha$ large. It is natural to ask if similar rates of convergence can be obtained in a similar setting for the dynamical system \eqref{dysy} (see, also, \cite[Section 5.4]{ACR}).

\item[$\bullet$] in the literature, in the context of dynamical systems, regularization terms have been considered not only in open-loop, but also in closed-loop form (see, for instance, \cite{ARS}). It is an interesting question if one can obtain for the dynamical system \eqref{dysy} similar results if the Tikhonov regularization term is taken in closed-loop form.

\item[$\bullet$] a natural question is to formulate proper numerical algorithms via time discretization of \eqref{dysy}, to investigate their theoretical convergence properties, and to validate them with numerical experiments.
\end{itemize}

{\bf Acknowledgements.} The authors are thankful to an anonymous reviewer for comments and remarks which improved the quality of the paper.

\appendix
\section{Appendix}\label{app}

In this appendix, we collect some lemmas and technical results which we will use in the analysis of the dynamical system \eqref{dysy}. The following lemma was stated for instance  in \cite[Lemma A.3]{ACR} and is used to prove the convergence of the objective function along the trajectory to its minimal value.

\begin{lemma}\label{nullimit} Let $\d>0$ and $f\in L^1((\d,+\infty),\R)$ be a nonnegative and continuous function.  Let $\varphi:[\d,+\infty)\To[0,+\infty)$ be a nondecreasing function such that $\lim_{t\To+\infty}\varphi(t)=+\infty.$ Then it holds
$$\lim_{t\To+\infty}\frac{1}{\varphi(t)}\int_\d^t\varphi(s)f(s)ds=0.$$
\end{lemma}

The following statement is the continuous counterpart of a convergence result
of quasi-Fej\'er monotone sequences. For its proofs we refer to \cite[Lemma 5.1]{abbas-att-sv}.

\begin{lemma}\label{fejer-cont1} Suppose that $F:[t_0,+\infty)\rightarrow\R$ is locally absolutely continuous and bounded from below and that
there exists $G\in L^1([t_0,+\infty), \R)$ such that
$$\frac{d}{dt}F(t)\leq G(t)$$
for almost every $t \in [t_0,+\infty)$. Then there exists $\lim_{t\To +\infty} F(t)\in\R$.
\end{lemma}

The following technical result is \cite[Lemma 2]{att-p-r-jde2016}.
\begin{lemma}\label{lemmader} Let $u : [t_0,+\infty)\To\mathcal{H}$ be a continuously differentiable function satisfying $u(t) +\frac{t}{\a}\dot{u}(t)\To u \in\mathcal{H}$ as $t\To+\infty,$ where $\a>0.$
 Then $u(t)\To u$ as $t \To+\infty.$
 \end{lemma}

The continuous version of  the Opial Lemma (see \cite{att-c-p-r-math-pr2018}) is the main tool for proving weak convergence for the generated trajectory.

\begin{lemma}\label{Opial}
Let $S \subseteq \mathcal{H}$ be  a nonempty set and $x : [t_0, +\infty) \to H$ a given map such that:
\begin{align*}
& (i) \quad \mbox{for every }z \in S \ \mbox{the limit} \ \lim\limits_{t \To +\infty} \| x(t) - z \|  \ \mbox{exists};\\
& (ii) \quad \text{every weak sequential limit point of } x(t) \text{ belongs to the set }S.
\end{align*}
Then the trajectory $x(t)$ converges weakly to an element in $S$ as $t \to + \infty$.
\end{lemma}


\begin{thebibliography}{999}
\bibitem{abbas-att-sv} B. Abbas, H. Attouch, B.F. Svaiter, {\it Newton-like dynamics and forward-backward methods for
structured monotone inclusions in Hilbert spaces}, Journal of Optimization Theory and its Applications 161(2), 331-360, 2014

\bibitem{alv-att-bolte-red} F. Alvarez, H. Attouch, J. Bolte, P. Redont, {\it A second-order gradient-like dissipative dynamical system
with Hessian-driven damping. Application to optimization and mechanics}, Journal de Math\'{e}matiques Pures et Appliqu\'{e}es (9) 81(8),
747--779, 2002

\bibitem {alv-cab} F. Alvarez, A. Cabot, {\it Asymptotic selection of viscosity equilibria of semilinear evolution equations by 
the introduction of a slowly vanishing term}, Discrete Continuous Dynamical Systems 15, 921--938, 2006

\bibitem{att} H. Attouch, {\it Viscosity solutions of minimization problems}, SIAM Journal on Optimization 6(3), 769--806, 1996

\bibitem{att-ch-fad-r-arx2019} H. Attouch, Z. Chbani, J. Fadili, H. Riahi, {\it First-order optimization algorithms via inertial systems with Hessian driven damping}, arXiv:1907.10536v1, 2019

\bibitem{att-ch-r-esaim2019} H. Attouch, Z. Chbani, H. Riahi, {\it Rate of convergence of the Nesterov accelerated gradient method in the subcritical case 
$\alpha\leq 3$}, ESAIM: Control, Optimisation and Calculus of Variations 25(2), 34pp., 2019

\bibitem{att-c-p-r-math-pr2018} H. Attouch, Z. Chbani, J. Peypouquet, P. Redont, {\it Fast convergence of inertial dynamics and algorithms with asymptotic
vanishing viscosity}, Mathematical Programming 168(1-2) Ser. B, 123--175, 2018

\bibitem{ACR} H. Attouch, Z. Chbani, H. Riahi, {\it Combining fast inertial dynamics for convex optimization with Tikhonov regularization}, Journal of Mathematical Analysis and Applications 457(2), 1065--1094, 2018

\bibitem{att-com1996} H. Attouch, R. Cominetti, {\it A dynamical approach to convex minimization coupling approximation with the steepest descent method},  Journal of Differential Equations 128(2), 519--540, 1996

\bibitem{AttouchCzarnecki} H. Attouch, M.-O. Czarnecki, {\it Asymptotic control and stabilization of nonlinear oscillators with non-isolated equilibria}, Journal of Differential Equations 197, 278--310, 2002

\bibitem{att-p-r-jde2016} H. Attouch, J. Peypouquet, P. Redont, {\it Fast convex optimization via inertial dynamics with Hessian driven damping}, Journal of Differential Equations 261(10), 5734--5783, 2016

\bibitem{ARS} H. Attouch, P. Redont, B.F. Svaiter, {\it Global convergence of a closed-loop regularized Newton method for solving monotone inclusions in Hilbert spaces}, Journal of Optimization Theory and Applications 157, 624--650, 2013

\bibitem{att-sv2011} H. Attouch, B.F. Svaiter, {\it A continuous dynamical Newton-like approach to solving monotone inclusions}, SIAM Journal on Control and Optimization 49(2), 574--598, 2011

\bibitem{cabotenglergadat} A. Cabot, H. Engler, S. Gadat, {\it On the long time behavior of second order differential equations with asymptotically small dissipation}, Transactions of the American Mathematical Society 361, 5983--6017, 2009

\bibitem {com-peyp-sor} R. Cominetti, J. Peypouquet, S. Sorin, {\it Strong asymptotic convergence of evolution equations governed 
by maximal monotone operators with Tikhonov regularization}, Journal of Differential Equations 245, 3753--3763, 2008

\bibitem{nesterov83} Y. Nesterov, {\it A method of solving a convex programming problem with convergence rate $O( 1 / k^{2})$}, Soviet Mathematics Doklady 27, 372--376, 1983

\bibitem{sellyou} G.R. Sell, Y. You, {\it Dynamics of Evolutionary Equations}, Springer, New York, 2002

\bibitem{shi-du-jordan-su2018} B. Shi, S.S. Du, M.I. Jordan, W.J. Su, {\it Understanding the acceleration phenomenon via high-resolution differential equations}, arXiv:1810.08907v3, 2018

\bibitem{shi-du-su-jordan2019} B. Shi, S.S. Du, W.J. Su, M.I. Jordan, {\it Acceleration via symplectic discretization of high-resolution differential equations}, Advances in Neural Information Processing Systems 32 (NIPS 2019), 5745--5753, 2019

\bibitem{su-boyd-candes} W. Su, S. Boyd, E.J. Cand\`{e}s, {\it A differential equation for modeling Nesterov's accelerated gradient method: theory and insights}, Journal of Machine Learning Research 17(153), 1--43, 2016
\end{thebibliography}
\end{document}